\documentclass[preprint,11pt]{elsarticle}
\usepackage[margin=1.in]{geometry}

\usepackage{mathtools}
\usepackage{amsmath}
\usepackage{amsfonts}
\usepackage{amsthm}
\usepackage{amssymb}
\usepackage[table,x11names]{xcolor}
\usepackage{tikz}
\usetikzlibrary{shapes.misc}

\usepackage{hyperref}

\usepackage[toc,page]{appendix}

\usepackage{graphicx}
\DeclareGraphicsExtensions{.pdf}

\DeclareMathOperator{\hyphen}{-}

\newcommand{\half}{\frac{1}{2}}
\newcommand{\todo}[1]{{}}

\newcommand{\real}[1]{\mathrm{Re}\left( #1 \right)}
\newcommand{\imag}[1]{\mathrm{Im}\left( #1 \right)}

\newtheorem{lemma}{Lemma}
\newtheorem{theorem}{Theorem}
\newtheorem{remark}{Remark}

\newcommand{\lastCell}{\ensuremath{I_N}}

\newcommand{\lastCellLinfy}{\ensuremath{L^\infty(\lastCell)}}
\newcommand{\lastCellWinfx}[1]{\ensuremath{W^{#1, \infty}(\lastCell)}}

\newcommand{\LinfC}{\ensuremath{C^*}}

\newcommand{\CoercivitySymbol}{\ensuremath{\alpha}}
\newcommand{\BoundednessSymbol}{\ensuremath{\gamma}}

\newcommand{\eDerivativeBound}[1]{\ensuremath{
    \left|e_1\right|
    \exp\left(\half |b|\right)
    2^{#1 + 2}
    \max\left(2, \dfrac{2}{|D|\Delta y}\right)
    |D|^{#1}
}}

\newcommand{\COne}{\ensuremath{
16
\left(
\max\left(1, \left(\dfrac{1}{2} b\right)^{m + 2}\right)
\exp\left(\half b L\right)
\right)^2
\max(\hat{C}_m, \hat{C}_{m + 1}) ((m + 3)!)^2
}}

\newcommand{\CTwo}{\ensuremath{
\max\left(
1,
(m + p + 1)\hat{C}(m + p, n) \left(\LinfC + \dfrac{1}{(m + p)!}\right)
\right)
}}

\newcommand{\CFour}{\ensuremath{
2 ((1 + m + p)(1 + \hat{C}(m + p, n)))
2^{m + p + 2}
\max\left(1, \sqrt{C_1}\right)
\exp\left(\half |b|\right)
}}

\newcommand{\CSixSymbolic}{\ensuremath{
C_6([0, L], \vec{b}, c, n, C(f), C_R(b_0), C_T)
}}

\newcommand{\CEightSymbolic}{\ensuremath{
C_8([0, L], f, C_7, \vec{b}, c, p, n)
}}

\tikzset{cross/.style={cross out, draw=black, minimum size=2*(#1-\pgflinewidth),
inner sep=0pt, outer sep=0pt},
cross/.default={2.5pt}}

\journal{ArXiv}

\begin{document}

\begin{frontmatter}

%% Title, authors and addresses

%% use the tnoteref command within \title for footnotes;
%% use the tnotetext command for theassociated footnote;
%% use the fnref command within \author or \address for footnotes;
%% use the fntext command for theassociated footnote;
%% use the corref command within \author for corresponding author footnotes;
%% use the cortext command for theassociated footnote;
%% use the ead command for the email address,
%% and the form \ead[url] for the home page:
%% \title{Title\tnoteref{label1}}
%% \tnotetext[label1]{}
%% \author{Name\corref{cor1}\fnref{label2}}
%% \ead{email address}
%% \ead[url]{home page}
%% \fntext[label2]{}
%% \cortext[cor1]{}
%% \address{Address\fnref{label3}}
%% \fntext[label3]{}

\title{Using \texorpdfstring{\(p\)}{\emph{p}}-Refinement to Increase Boundary Derivative Convergence Rates}

%% Note the weird placement of corref and fnref so that the right parenthesis
%% shows up in the correct place
\author[rpi]{David\corref{cor1}\fnref{rtgthanks} Wells}
\ead{wellsd2@rpi.edu,daverwells@gmail.com}
\author[rpi]{Jeffrey Banks}
\ead{banksj3@rpi.edu}

\cortext[cor1]{Corresponding author.}
\address[rpi]{Department of Mathematical Sciences,
Rensselaer Polytechnic Institute, Troy, NY USA}

\fntext[rtgthanks]{This work was supported in part by the NSF through grant
DMS-1344962.}

\begin{abstract}
    Many important physical problems, such as fluid structure interaction or
    conjugate heat transfer, require numerical methods that compute boundary
    derivatives or fluxes to high accuracy. This paper proposes a novel approach
    to calculating accurate approximations of boundary derivatives of elliptic
    problems. We describe a new continuous finite element method based on
    \(p\)-refinement of cells adjacent to the boundary that increases the local
    degree of the approximation. We prove that the order of the approximation on
    the \(p\)-refined cells is, in 1D, determined by the rate of convergence at
    the mesh vertex connecting the higher and lower degree cells and that this
    approach can be extended, in a restricted setting, to 2D problems. The
    proven convergence rates are numerically verified by a series of experiments
    in both 1D and 2D. Finally, we demonstrate, with additional numerical
    experiments, that the \(p\)-refinement method works in more general
    geometries.
\end{abstract}


\end{frontmatter}

%% \linenumbers

%% main text
\section{Introduction}
    \label{sec:introduction}
    Simulation of many important physical problems, such as fluid structure
    interaction and conjugate heat transfer, requires numerical methods that
    compute boundary derivatives or fluxes to high accuracy. In some
    circumstances the only desired result of a calculation is a quantity derived
    from the boundary derivatives, such as a flux or stress: this problem has
    long been recognized as one of importance, and a variety of methods (see,
    e.g., \cite{CareyApproximateBoundaryFlux84, GalerkinProcedureFluxBoundary,
    WheelerGalerkinProcedureFluxBVP1973}) have been proposed that allow
    reconstruction of an accurate boundary flux from less accurate interior
    data. Accurate boundary derivatives are also required for some numerical
    boundary conditions. For example, in \cite{LiStablePartitionedFSIBeams} the
    authors presented a new discrete boundary condition for a fluid-structure
    interaction problem based on matching accelerations, instead of velocities,
    and obtained a traction boundary condition involving second derivatives of
    the fluid velocity. This boundary condition was the key ingredient in a new
    partitioned algorithm that was high-order, partitioned, and stable without
    subiterations. While standard in the finite difference community (see, e.g.,
    \cite{BanksAMPCompressibleNonlinear, LiStablePartitionedFSIBeams}) these
    equations, usually called \emph{compatibility boundary conditions}, are not
    commonly used in finite element methods, though they have appeared in
    some recent work \cite{banksGalerkinDifference}.

    A variety of algorithms have been proposed for calculating higher order
    derivative values from lower order data calculated by a finite element
    method (see, e.g., \cite{CareyDerivativeFromFEM1982,
    QuasiProjectionAnalysis1978, WheelerGalerkinProcedureFluxBVP1973,
    ZhangNewGradientRecovery2005, ZienkiewiczVolume1}): most of these algorithms
    rely on \emph{data post-processing}, where one uses least squares or other
    fitting procedure to fit a higher-degree polynomial through known
    superconvergence points, as discussed in
    \cite{BabuskaStrouboulisFEReliability}. Another class of methods relies on
    the application of high-order finite difference stencils to data derived on
    either a uniform or quasi-uniform grid \cite{GuoHessianRecovery2016}. A
    common feature of several postprocessing techniques is that they require a
    grid satisfying some smoothness condition: without such a condition,
    the error in the solution may be dominated by pollution error from
    grid irregularities; see Chapter 4 of \cite{BabuskaStrouboulisFEReliability}
    for additional information on the impact of grid regularity. In particular,
    of the three most common versions of the finite element method
    (\(h\)-refinement based, \(p\)-refinement based, and \(hp\)-refinement
    based) these postprocessing methods are almost always based on estimates
    from the \(h\)-refinement version.

    This paper proposes a novel alternative to current techniques. We present a
    boundary cell \emph{\(p\)-refinement} (i.e., locally increasing the degree
    of the approximation space) strategy to improve the accuracy of boundary
    derivatives instead of postprocessing the solution. The numerical
    experiments in Section \ref{sec:numerical-results} use Lagrange
    \(p\)-refinement to increase the local approximation degree: a possible
    alternative to this is to add degrees of freedom corresponding to normal
    derivatives on the boundary. This \(p\)-refinement results in higher rates
    of convergence in the normal derivatives along the boundary. The theoretical
    results are based on the two-dimensional linear
    convection-diffusion-reaction problem
    \begin{equation}
        \label{eq:cdr-equation}
        -\Delta u + \vec{b} \cdot \nabla u + c u = f
    \end{equation}
    with homogeneous Dirichlet boundary conditions in \(y\), periodic boundary
    conditions in \(x\), normalized viscosity, constant advection velocity
    \(\vec{b}\), constant reaction rate \(c > 0\) (which is the standard
    well-posedness assumption; see Lemma 5.1 in \cite{WellsThesis2015} or
    Chapters 3 and 4 of \cite{Roos2008robust} for further discussion and
    justification), and forcing \(f\).
    \begin{figure}
        \centering
        \include{p-refine}
        \caption{Two different implementations of \(p\)-refinement for boundary
        cells adjacent to interior bilinear cells, where the finite element
        spaces are chosen as nodal interpolants. The diagram on the left is of
        \(Q^1\) elements adjacent to \(Q^4\) elements: the degrees of freedom
        with support points along the two common faces would ordinarily be
        constrained in a way that makes the solution continuous. The scheme
        proposed in Section \ref{sec:extension-to-2d} uses a similar procedure
        to constrain \emph{all} such nonnormal degrees of freedom on each
        boundary cell, effectively reducing the local approximation space
        to tensor products of \(P^1(x)\) and \(P^4(y)\). Since the degree
        of the approximation in the normal direction determines the derivative
        convergence rates, one could obtain the same effect by adding degrees of
        freedom corresponding to normal derivatives on the boundary instead of
        doing Lagrange \(p\)-refinement.}
        \label{fig:show-nonnormal-elimination}
    \end{figure}

    Our main goal is to improve the accuracy of the finite element approximation
    to Equation \eqref{eq:cdr-equation}'s normal derivatives along the Dirichlet
    boundary, notated as \(\partial \Omega\). We will improve the accuracy by
    performing \(p\)-refinement in the direction normal to \(\partial \Omega\).
    \emph{\(p\)-refinement in the normal direction} means that, rather
    than standard \(Q^m\) elements, the polynomial space on each boundary cell
    is \(P^m \otimes P^{m + p}\); i.e., a tensor product between degree \(m\)
    polynomials in the tangential direction and degree \(m + p\) polynomials in
    the normal direction. Notate these spaces by \(Q^{(m, m + p)}\),
    where \(p \geq 0\). An example implementation of this type of
    \(p\)-refinement is shown on the right in Figure
    \ref{fig:show-nonnormal-elimination}, with numerical results shown in Figure
    \ref{fig:square-periodic-domain-convergence-motivation}. The error estimates
    proven in this paper rely on the tensor product discretization and, as a
    result, provide convergence rates at mesh vertices. Hence, the numerical
    experiments in Section \ref{sec:numerical-results} show convergence rates
    computed in the \(H^1{\hyphen}B\) and \(H^2{\hyphen}B\) seminorms, which are
    defined in terms of the gradient, Hessian, and normal derivatives along the
    non-periodic boundary:
    \begin{align}
        \label{eq:h1-b-error-definition}
        |u - u^h|_{H^1{\hyphen}B}
        &=
        \max_{i, j}
        \left|
        (\nabla (u - u^h) \cdot \vec{n})(\delta_i, \delta_j)
        : (\delta_i, \delta_j)
        \text{ are cell vertices on the boundary}
        \right|                                                               \\
        \label{eq:h2-b-error-definition}
        |u - u^h|_{H^2{\hyphen}B}
        &=
        \max_{i, j}
        \left|
        (\vec{n}^T \nabla^2 (u - u^h) \vec{n})(\delta_i, \delta_j)
        : (\delta_i, \delta_j)
        \text{ are cell vertices on the boundary}
        \right|.
    \end{align}

    \begin{figure}[!htb]
        \centering
        \includegraphics[width=3.2in]{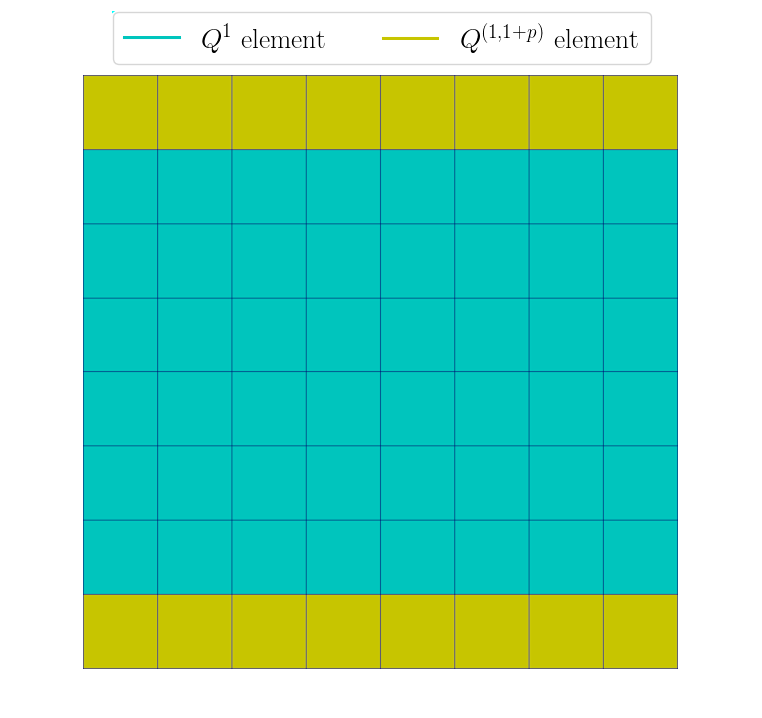}
        \includegraphics[width=3.2in]{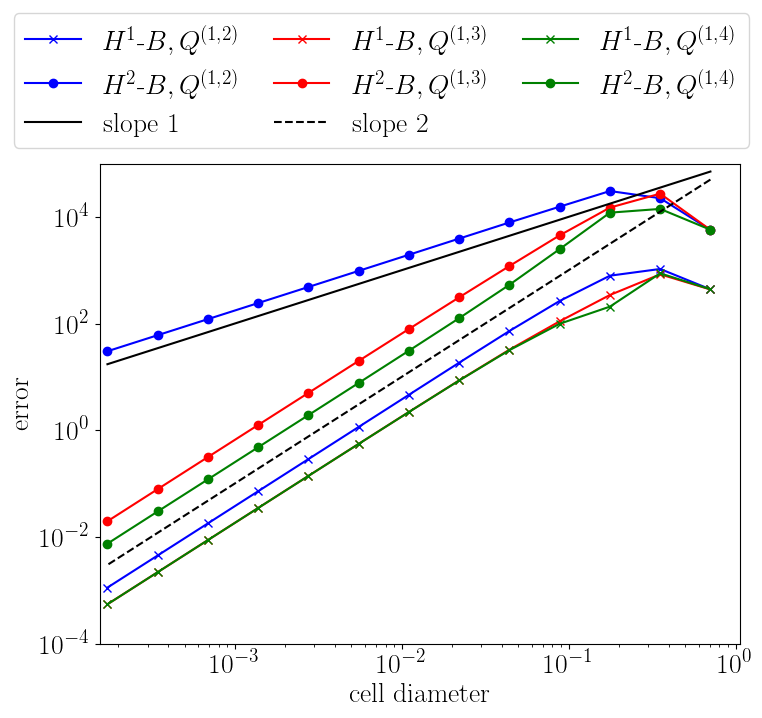}

        \caption{Convergence rates for a numerical approximation of
        \eqref{eq:cdr-equation} on a domain with periodic boundary
        conditions in the \(x\) direction and Dirichlet boundary conditions in
        the \(y\) direction. The grid on the left depicts which cells have been
        \(p\)-refined; the plot on the right shows convergence rates in the
        seminorms defined by Equations
        \eqref{eq:h1-b-error-definition}-\eqref{eq:h2-b-error-definition}. The
        cyan cells have \(Q^1\) (bilinear) shape functions; the yellow cells
        have been \(p\)-refined in the normal direction (i.e., they are \(P^m
        \otimes P^{m + p}\) elements, notated as \(Q^{(m, m + p)}\)).}
        \label{fig:square-periodic-domain-convergence-motivation}
    \end{figure}

    The proposed \(p\)-refinement method does not neatly fit into the usual
    taxonomy of the three common versions of the finite element method. The
    proven convergence rates depend on the cell diameters but use multiple
    (fixed) polynomial orders in the domain, which resembles \(p\)-refinement:
    however, this is not a \(p\)-refinement method since the refinement in \(p\)
    is static (i.e., only cells adjacent to the nonperiodic boundary are
    \(p\)-refined, and the degrees \(m\) and \(m + p\) are fixed during grid
    refinement) and not dependent on any a-posteriori error or regularity
    estimator. As such, this method is not well described by the standard
    \(hp\)-finite element error estimate \cite{AinsworthHPAspects}
    \begin{equation}
        \label{eq:standard-hp-convergence}
        \|u - u^{hp}\|_{H^1(\Omega)} \leq
        C h^\mu p^{-(m - 1)} \|u\|_{H^m(\Omega)}
    \end{equation}
    for elliptic problems in an energy norm, where \(\mu = \min(p, m - 1)\),
    \(p\) is the polynomial order, and the grid is quasiuniform with cell
    diameter \(h\). Additionally, this method is not well-described by the
    classic \(p\)-refinement estimate
    \begin{equation}
        \|u - u^h\|_{H^1} \leq C p^{-(m - 1)} \|u\|_{H^m}
    \end{equation}
    given in \cite{BabuskaOptimalConvergencePVersion} since the degree
    \(p\) is fixed during the refinement process.

    The rest of this paper is organized as follows: Section
    \ref{sec:oned-analysis} presents the essential theory for a one-dimensional
    version of the proposed scheme, including proofs showing that performing
    \(p\)-refinement on boundary cells improves the boundary derivative
    convergence rates. Section \ref{sec:extension-to-2d} extends these
    results to higher dimensions: for structured grids and suitable finite
    elements, one can recover essentially the same convergence results from the
    1D case at mesh vertices along the boundary of the domain. Finally,
    Section \ref{sec:numerical-results} summarizes some numerical
    experiments that demonstrate the proven asymptotic convergence rates and
    show that the results are still valid in a simple but non-Cartesian
    geometry. Section \ref{sec:conclusions} concludes with a summary of the
    presented results and some conjectures regarding possible extensions of this
    work.

\section{One-dimensional analysis of the model problem}
    \label{sec:oned-analysis}
    \subsection{Introduction}
        The numerical scheme for the model problem described in Section
        \ref{sec:introduction} may be analyzed, in part, by applying an analog
        of the discrete Fourier transform in the periodic direction to reduce it
        to a one-dimensional problem. Therefore, we begin our analysis of
        the discretization of Equation \eqref{eq:cdr-equation} by analyzing the
        simpler model
        \begin{equation}
            \label{eq:cdr-equation-oned}
            -u_{yy} + b u_y + \tilde{c} u = \tilde{f}(y)
        \end{equation}
        with homogeneous Dirichlet boundary conditions and, having
        applied a discrete Fourier transform, \(\real{\tilde{c}} > 0\) and
        \(\tilde{c}\) and \(\tilde{f}(y)\) may be complex. The primary result of
        this section is an analysis of the \(p\)-refinement scheme in a single
        space dimension for Equation \eqref{eq:cdr-equation-oned}.

        It is well known (a full proof is given in
        \cite{TwoPointBVPConvergence}) that, for Equation
        \eqref{eq:cdr-equation-oned}, the rate of convergence of a finite
        element approximation consisting of continuous piecewise polynomials of
        order \(m\) at a mesh vertex (a point connecting two cells) \(y_j\) is
        \begin{equation}
            \label{eq:douglas-dupont-vertex-rate}
            |u^h(y_j) - u^h(y_j)| = O(h^{2 m}),
        \end{equation}
        where \(m\) is the degree of the polynomial space used in the two
        adjacent cells. Subsection
        \ref{subsec:vertex-adjacent-boundary-convergence} presents a special case
        of this theorem for the last interior vertices (i.e., the interior vertices of
        the two boundary cells), showing that the approximation gains one
        additional order of accuracy at these two points. This is significantly
        better than the standard superconvergence result at the Gauss-Lobatto
        points of the function value of (see, e.g., the table in Section 1.10 of
        \cite{WahlbinSuperconvergence})
        \begin{equation}
            \label{eq:general-lobatto-superconvergence}
            |u^h(y_g) - u^h(y_g)| = O(h^{m + 2}).
        \end{equation}
        The authors of \cite{TwoPointBVPConvergence} note that one could perform
        local \(p\)-refinement and achieve higher order accuracy on a specific
        cell due to the higher convergence order at the mesh vertices. This is the
        primary idea used to achieve higher order derivative boundary
        convergence. The proofs of these results rely on computations with
        Greens' functions and, as such, do not have immediate extensions to
        higher dimensions due to the nonintegrability of higher-dimension
        Greens' functions (see the discussion regarding Greens' functions in
        \cite{BabuskaStrouboulisFEReliability} for additional information on the
        limitations of this approach).

    \subsection{Well-posedness of the system}
        This subsection presents some basic analysis of Equation
        \eqref{eq:cdr-equation-oned}. Assuming that the solution is
        complex-valued since the forcing and low-order coefficient may be
        complex, let \(H^1_0([0, L])\) be the Hilbert space of complex-valued
        functions whose derivatives and function values are square-integrable on
        \([0, L]\) and have a value of zero at \(0\) and \(L\). Consider
        the sesquilinear and skew-linear forms associated with Equation
        \eqref{eq:cdr-equation-oned}:
        \begin{align}
            \label{eq:cdr-complex-sesquilinear}
            a(\phi, \psi)
            &=
            \int_0^L \phi_y \bar{\psi}_y dy
            +
            b \int_0^L \phi_y \bar{\psi} dy
            +
            \tilde{c}
            \int_0^L \phi \bar{\psi} dy                                       \\
            \label{eq:cdr-system-skew-linear}
            l(\psi)
            &=
            \int_0^L \tilde{f} \bar{\psi} dy.
        \end{align}
        The weak problem is, for a Hilbert space \(X \subseteq H_0^1([0, L])\),
        finding \(z \in X\) such that, for all \(\psi \in X\)
        \begin{equation}
            \label{eq:cdr-complex-problem}
            a(z, \psi) = l(\psi).
        \end{equation}

        \begin{theorem}
            \label{thm:complex-oned-wellposedness}
            The weak problem given by Equation \eqref{eq:cdr-complex-problem}
            and a Hilbert space \(X \subseteq H^1_0([0, L])\) is well-posed when
            \(\real{\tilde{c}} > 0\).
        \end{theorem}

        \begin{proof}
            This theorem immediately follows from the complex-valued version of
            the Lax-Milgram Theorem; see Chapter 6, Theorem 6 of
            \cite{LaxFunctionalAnalysis}. Since
            \begin{subequations}
                \begin{align}
                    |a(\phi, \psi)|
                    &=
                    \left|
                    \int_0^L \phi_y \bar{\psi}_y dy
                    +
                    b \int_0^L \phi_y \bar{\psi} dy
                    +
                    \tilde{c}
                    \int_0^L \phi \bar{\psi} dy
                    \right|                                                   \\
                    &\leq
                    \BoundednessSymbol
                    \|\phi\|_{H^1}
                    \|\psi\|_{H^1}, \forall \phi, \psi \in H^1_0([0, L]),
                \end{align}
            \end{subequations}
            the sesquilinear form is bounded with boundedness constant
            \begin{equation}
                \label{eq:cdr-boundedness-constant}
                \BoundednessSymbol = (1 + |b| + |\tilde{c}|).
            \end{equation}
            To show coercivity, decompose \(\phi(y)\) and \(\tilde{c}\)
            into real and imaginary components
            \begin{equation}
                \phi = \phi_R + \phi_I I
                \text{ and }
                \tilde{c} = \real{\tilde{c}} + \imag{\tilde{c}} I
            \end{equation}
            where, to avoid confusion with the mesh vertex index \(i\), notate
            \(\sqrt{-1} = I\) and obtain, again \(\forall \phi \in H^1_0([0,
            L])\),
            \begin{subequations}
                \begin{align}
                    |a(\phi, \phi)|
                    &\geq
                    |\mathrm{Re}(a(\phi, \phi))|                              \\
                    &=
                    \left|
                    \left\|\phi_y\right\|_{L^2}^2 +
                    b\int_0^L
                    \left(\phi_R \phi_{R,y} + \phi_I \phi_{I, y}\right) dy +
                    \real{\tilde{c}} \left\|\phi\right\|_{L^2}^2
                    \right|                                                   \\
                    &\geq
                    \CoercivitySymbol
                    \|\phi\|_{H^1}^2
                \end{align}
            \end{subequations}
            due to the assumption of homogeneous boundary conditions with
            coercivity constant
            \begin{equation}
                \label{eq:cdr-coercivity-constant}
                \CoercivitySymbol = \min(1, \real{\tilde{c}}).
            \end{equation}
            Applying the Lax-Milgram Theorem yields the stated result.
        \end{proof}

    \subsection{Superconvergence at vertices near the boundary}
        \label{subsec:vertex-adjacent-boundary-convergence}
        The primary result of this section, Theorem
        \ref{thm:last-interior-vertex-superconvergence}, depends on the rate of
        convergence of the interpolation of a smooth function on a single cell
        in a finite element discretization: in particular, the scaling of the
        derivatives of the interpolated function and the diameter of the cell
        play a key role in Theorem
        \ref{thm:last-interior-vertex-superconvergence}. Hence, this
        subsection begins with a special case of the Bramble-Hilbert lemma
        (see, e.g., \cite{CiarletBook} for a proof of the general case):

        \begin{lemma}
            \label{lem:interpolation-h1}
            %% define K
            Let \(K = [a, a + \Delta y]\) be an interval with diameter \(\Delta
            y\).
            %% define spaces
            Let \(v \in V\) where \(V = W^{m + 1, \infty}(K)\). Consider a
            finite element space \(V^h \subset H^1(K), V^h = P^m(K)\), where
            %% define P
            \(P^m(K)\) is the space of degree \(m\) polynomials on the interval
            \(K\).
            %% define interpolation
            Define the interpolation operator \(\Pi : V \rightarrow
            V^h\) such that
            \begin{equation}
                \label{eq:cellwise-interpolation}
                (\Pi u)(\xi_j) = u(\xi_j), \Pi u \in V^h
                \text{ where }
                \xi_j = a + \dfrac{j}{m} \Delta y, j = 0, 1, \cdots, m.
            \end{equation}
            Then \(v^h = \Pi v\) satisfies the error estimate
            \begin{equation}
                \| v - v^h \|_{H^1(K)}^2
                \leq
                \hat{C} \Delta y^{2 m + 1}
                \| v^{(m + 1)} \|_{L^\infty(K)}^2
            \end{equation}
            where \(\hat{C}\) is independent of \(\Delta y\) and \(v\).
        \end{lemma}

        \begin{proof}
            This lemma follows immediately from applying Theorem 3.1.5 of
            \cite{CiarletBook} with parameters \(m = 0\), \(q = 2\), and \(p =
            \infty\) and \(m = 1\), \(q = 2\), and \(p = \infty\) and then
            summing the squares of the results.
        \end{proof}

        \begin{remark}
            The same result holds for more general interpolation operators that
            also interpolate derivative values of the function (i.e., for
            Hermite-type elements), or for nonuniform point distributions (e.g.,
            using the Gauss-Lobatto points as nodes).
        \end{remark}

        Theorem \ref{thm:last-interior-vertex-superconvergence} is a special
        case for the mesh vertex located at \(L - \Delta y\) (i.e., the last
        interior vertex) of the more general vertex convergence rate proven in
        \cite{TwoPointBVPConvergence}.
        \begin{theorem}
            \label{thm:last-interior-vertex-superconvergence}
            Consider a partition of \([0, L]\) into \(N\) cells of equal
            diameter \(\Delta y\).
            %% FE space notation
            Let \(V^h \subset H^1_0([0, L])\) be the finite element space of
            piecewise continuous polynomials of degree \(m\) on interior cells
            and degree \(m + p\) on boundary cells, where \(p \geq 1\).
            %% interpolation operator
            Define the interpolation operator \(\Pi : H_0^1([0, L]) \rightarrow
            V^h\) on each cell in the same way as Equation
            \eqref{eq:cellwise-interpolation}, where interior cells have \(m +
            1\) interpolation points and boundary cells have \(m + p + 1\)
            interpolation points (including, in both cases, cell vertices).
            %% back to the problem
            Let \(u^h(y) \in V^h\) be the solution to
            \eqref{eq:cdr-complex-problem} with \(X = V^h\). Let \(u(y) \in
            W^{m + 1, \infty}([0, L])\) be the solution of Equation
            \eqref{eq:cdr-complex-problem} with \(X = W^{m + 1, \infty}([0,
            L])\). Assume that
            \begin{equation}
                \label{eq:cdr-D-definition}
                D = \dfrac{\sqrt{b^2 + 4 \tilde{c}}}{2}, \real{D} \geq 1,
                \text{ and } \real{D} \geq |\imag{D}|.
            \end{equation}
            To simplify some inequalities, assume that
            \begin{subequations}
                \begin{align}
                    \label{eq:b-sign-assumption}
                    b &\geq 0                                                 \\
                    L &\geq 1.
                \end{align}
            \end{subequations}
            Then there exists a constant \(C_1(m, p) = C_1\) dependent only on
            \(m\) and \(p\) such that
            \begin{equation}
                | u(L - \Delta y) - u^h(L - \Delta y)|
                \leq
                \dfrac{\BoundednessSymbol^2}{\CoercivitySymbol}
                |D|^{m + 1}
                \sqrt{C_1} \|u^{(m + 1)}\|_{L^\infty([0, L])}
                \Delta y^{2 m + 1}
            \end{equation}
            where \(\BoundednessSymbol\) and \(\CoercivitySymbol\) are the
            boundedness and coercivity constants defined in Equations
            \eqref{eq:cdr-boundedness-constant} and
            \eqref{eq:cdr-coercivity-constant}.
        \end{theorem}

        \begin{proof}
            This proof relies on some elementary inequalities that are
            true due to the bounds on \(|D|\) and \(L\) (for proofs, see
            Appendix \ref{sec:greens-inequalities}):
            \begin{subequations}
                \begin{align}
                    \label{eq:o-of-dy-factor}
                    \left|
                    \dfrac{\exp(-2 D \Delta y) - 1}{2 D}
                    \right|
                    &\leq \Delta y                                            \\
                    \label{eq:exp-sum-over-exp-bound}
                    \left|
                    \dfrac{\exp(2 D L) \pm \exp(2 D \Delta y)}
                    {\exp(2 D L) - 1}
                    \right|
                    &\leq 4                                                   \\
                    \label{eq:exp-sum-over-D-bound}
                    \left|
                    \dfrac
                    {
                    \exp(-D(L + \Delta y - y)) \pm
                    \exp(D(L - \Delta y - y))
                    }
                    {2 D}
                    \right|
                    &\leq
                    \dfrac{1}{|D|}.
                \end{align}
            \end{subequations}

            This proof follows the outline of the classic superconvergence
            result in \cite{TwoPointBVPConvergence}. Consider the Greens'
            function \(G(y)\) associated with the operator implied by
            \eqref{eq:cdr-equation-oned} and the point force \(\delta\) centered
            at \(L - \Delta y\). This Greens' function is defined by the
            following equations:
            \begin{align}
                -G_{yy}(y) + b G_y(y) + \tilde{c} G(y)
                &= \delta(y - (L - \Delta y))                                 \\
                G(0) &= 0                                                     \\
                [G(L - \Delta y)] &= 0                                        \\
                [G_y(L - \Delta y)] &= 1                                      \\
                G(L) &= 0
            \end{align}
            i.e., \(G(y)\) has homogeneous boundary conditions, is continuous at
            \(L - \Delta y\), and has a jump in its derivative at \(L - \Delta
            y\). This implies, by standard Greens' function calculations (see,
            e.g., \cite{HabermanPDEs}) that
            \begin{equation}
                \label{eq:greens-function-definition}
                G(y)
                =
                \begin{cases}
                    {
                    G^L(y) =
                    \left(
                    c_{1} \cosh\left(D y\right) +
                    c_{2} \sinh\left(D y\right)
                    \right)
                    } \exp{\left(\frac{1}{2} \, b y\right)}
                    & 0 \leq y \leq L - \Delta y                              \\
                    {
                    G^R(y) =
                    \left(
                    c_{3} \cosh\left(D y\right) +
                    c_{4} \sinh\left(D y\right)
                    \right)
                    } \exp{\left(\frac{1}{2} \, b y\right)}
                    & L - \Delta y < y \leq L.
                \end{cases}
            \end{equation}
            Enforcing the four conditions listed above yields
            \begin{subequations}
                \begin{align}
                    \label{eq:green-c1}
                    c_{1} &= 0                                                \\
                    c_{2} &= -\frac
                    {
                    \cosh\left(A_{2}\right) \sinh\left(D L\right) -
                    \cosh\left(D L\right) \sinh\left(A_{2}\right)
                    }
                    {A_{1} D \sinh\left(D L\right)}                           \\
                    c_{3} &= -\frac {\sinh\left(A_{2}\right)}{A_{1} D}        \\
                    \label{eq:green-c4}
                    c_{4} &= \frac
                    {
                    \cosh\left(D L\right) \sinh\left(A_{2}\right)
                    }{A_{1} D \sinh\left(D L\right)}
                \end{align}
            \end{subequations}
            where
            \begin{subequations}
                \begin{align}
                    \label{eq:green-a1}
                    A_1 &=
                    \exp{\left(
                    \half b (L - \Delta y)
                    \right)}                                                  \\
                    \label{eq:green-a2}
                    A_2 &=  {\left(L - {\Delta y}\right)} D.
                \end{align}
            \end{subequations}
            Define product decompositions \(G^L(y) = G^L_1(y) G^L_2(y)\) and
            \(G^R(y) = G^R_1(y) G^R_2(y)\), where
            \begin{subequations}
                \begin{align}
                    G^L_1(y) &= c_2 \sinh(D y)                                \\
                    G^L_2(y) &= \exp\left(\half b y\right)                    \\
                    G^R_1(y) &= c_3 \cosh(D y) + c_4 \sinh(D y)               \\
                    G^R_2(y) &= \exp\left(\half b y\right).
                \end{align}
            \end{subequations}
            As
            \begin{equation}
                \dfrac{d^n}{dy^n} G^L_1(y) =
                \begin{cases}
                    {
                    c_2 D^{n} \sinh(D y)
                    }
                    & n \text{ is even}                                       \\
                    {
                    c_2 D^{n} \cosh(D y)
                    }
                    & n \text{ is odd}                                        \\
                \end{cases}
            \end{equation}
            the moduli of derivatives of \(G^L_1(y)\) are maximized at \(y = L -
            \Delta y\) by the maximum modulus principle and the bound on the
            imaginary part
            %% TODO could we make this argument more rigorous?
            given by the assumption in Equation \eqref{eq:cdr-D-definition}.
            Similarly, by Assumption \eqref{eq:b-sign-assumption}, \(G^L_2(y)\)
            and all of its derivatives are maximized at the same point. Putting
            these together, derivatives of \(G_1^L(y)\) are bounded by
            \begin{subequations}
                \begin{align}
                    \left|\dfrac{d^n}{dy^n} G^L_1(y)\right|_{y = L - \Delta y}
                    &=
                    \left|
                    D^n
                    \dfrac{\exp(2 D L) - (-1)^n\exp(2 D \Delta y)}{\exp(2 D L) - 1}
                    \dfrac{\exp(-2 D \Delta y) - 1}{2 D}
                    \exp\left(\half b\left(\Delta y - L\right)\right)
                    \right|                                                   \\
                    &\leq
                    4 |D|^n
                    \Delta y
                    \exp\left(\half b\left(\Delta y - L\right)\right)         \\
                    &\leq
                    4 |D|^n
                    \Delta y
                \end{align}
            \end{subequations}
            by Equation \eqref{eq:o-of-dy-factor} and Equation
            \eqref{eq:exp-sum-over-exp-bound}. Notate the bound on derivatives
            of \(G^L_2(y)\) and \(G^R_2(y)\) as
            \begin{equation}
                \Gamma(b, L, n) =
                \Gamma =
                \max\left(1, \left(\dfrac{1}{2} b\right)^{n}\right)
                \exp\left(\half b L\right).
            \end{equation}
            Since \(1 \leq |D|\), the Leibniz rule provides a max
            norm estimate for derivatives of \(G^L(y)\):
            \begin{align}
                \left\|
                \dfrac{d^n}{dy^n} G^L(y)
                \right\|_{L^\infty([0, L - \Delta y])}
                &=
                \left\|
                \sum_{i = 0}^{n} \binom{n}{i}
                \left(
                \dfrac{d^{n - i}}{dy^{n - i}}
                G^L_1(y)
                \right)
                \left(
                \dfrac{d^{i}}{dy^{i}}
                G^L_2(y)
                \right)
                \right\|_{L^\infty([0, L - \Delta y])}                        \\
                &\leq
                \label{eq:g-minus-derivative-scaling}
                %% binomial term and 4 from G_1
                4 (n + 1)!
                %% maximal G_1 coefficient
                |D|^n
                %% maximal G_2 coefficient
                \Gamma
                %% the result looks nicer with dy at the end, but this comes
                %% from G_1
                \Delta y.
            \end{align}
            Similarly, for \(G^R_1(y)\) and \(L - \Delta y \leq y \leq L\)
            \begin{subequations}
                \begin{align}
                    \left|\dfrac{d^n}{dy^n} G^R_1(y)\right|
                    &=
                    \bigg|
                    D^{n}
                    \dfrac{\exp(2 D L) - \exp(2 D \Delta y)}{\exp(2 D L) - 1}
                    \dfrac{\exp(-(L + \Delta y - y) D) - (-1)^n\exp((L - \Delta y - y) D)}{2 D}\\
                    \nonumber
                    &\phantom{ = \bigg|}
                    \exp\left(\half b (\Delta y - L)\right)
                    \bigg|                                                    \\
                    &\leq
                    4 |D|^{n - 1}.
                \end{align}
            \end{subequations}
            Hence, application of the Leibniz rule provides
            \begin{align}
                \left\|\dfrac{d^n}{dy^n} G^R(y)\right\|_{L^\infty([L - \Delta y,
                L])}
                &\leq
                \left\|
                \sum_{i = 0}^{n} \binom{n}{i}
                \left(
                \dfrac{d^{n - i}}{dy^{n - i}}
                G^R_1(y)
                \right)
                \left(
                \dfrac{d^{i}}{dy^{i}}
                G^R_2(y)
                \right)
                \right\|_{L^\infty([L - \Delta y, L])}                        \\
                \label{eq:g-plus-derivative-scaling}
                &\leq
                4 (n + 1)! |D|^{n - 1}
                \Gamma.
            \end{align}
            Since \(G\) is the Greens' function
            \begin{subequations}
                \begin{align}
                    |u^h(L - \Delta y) - u(L - \Delta y)|^2
                    &= |(\delta(y - (L - \Delta y)), u^h - u)|^2              \\
                    &= |a(G, u^h - u)|^2                                      \\
                    &= |a(G - v^h, u^h - u)|^2
                \end{align}
            \end{subequations}
            for any \(v^h \in V^h\) by Galerkin orthogonality. Let \(v^h = \Pi
            G\) be the finite element interpolation of \(G\). By continuity of
            the sesquilinear form, application of Lemma
            \eqref{lem:interpolation-h1} to every cell and summing the result,
            the scalings in Equation \eqref{eq:g-minus-derivative-scaling} and
            Equation \eqref{eq:g-plus-derivative-scaling}, and C\'{e}a's lemma
            (see, e.g., Theorem 2.4.1 in \cite{CiarletBook})
            \begin{subequations}
                \begin{align}
                    |a(G(y) - v^h, u^h - u)|^2
                    &\leq
                    \BoundednessSymbol^2 \|u^h - u\|_{H^1}^2 \|G(y) - v^h\|_{H^1}^2\\
                    &\leq
                    \BoundednessSymbol^2 \|u^h - u\|_{H^1}^2
                    \Bigg[
                    \left(
                    \sum_{j = 0}^{N - 2}
                    \hat{C}_m
                    \|G^{L,(m + 1)}\|_{L^\infty([j \Delta y, (j + 1)\Delta y])}^2
                    \Delta y^{2 m + 1}
                    \right)
                    +                                                         \\
                    \nonumber
                    &\hphantom{{}\leq{} \BoundednessSymbol^2 \|u^h - u\|^2_{H^1}
                    \bigg\lbrack}
                    \hat{C}_{m + 1}
                    \|G^{R,(m + 2)}\|_{L^\infty([L - \Delta y, L])}^2
                    \Delta y^{2 (m + 1) + 1}
                    \Bigg]                                                    \\
                    &\leq
                    \BoundednessSymbol^2
                    \|u^h - u\|_{H^1}^2
                    \bigg[
                    \max(\hat{C}_m, \hat{C}_{m + 1})
                    \left(
                    4 (m + 3)! |D|^{m + 1}
                    \Gamma
                    \right)^2                                                 \\
                    \nonumber
                    &\hphantom{\leq\gamma \|u^h - u\|_{H^1}^2 \bigg\lbrack}
                    \bigg(
                    \Delta y^2 \left(\sum_{j = 0}^{N - 2} \Delta y^{2 m + 1}\right)
                    + \Delta y^{2(m + 1) + 1}
                    \bigg)\bigg]                                              \\
                    &\leq
                    \dfrac{\BoundednessSymbol^4}{\CoercivitySymbol^2}
                    \left\| \dfrac{d^{m + 1}}{dy^{m + 1}} u\right\|^2_{L^\infty([0, L])}
                    \Delta y^{2 m}
                    \bigg[
                    \max(\hat{C}_m, \hat{C}_{m + 1})
                    \left(4 (m + 3)! |D|^{m + 1} \Gamma\right)^2              \\
                    \nonumber
                    &\hphantom{\leq \BoundednessSymbol
                    %% phantom stuff
                    \left(
                    \dfrac{\BoundednessSymbol^2}{\CoercivitySymbol^2}
                    \left\|\dfrac{d^{m + 1}}{dy^{m + 1}} u\right\|^2_{L^\infty([0, L])}
                    \Delta y^{2 m}
                    \right)\bigg\lbrack}
                    %% no more phantom stuff
                    \left(\Delta y^{2 m + 2} + \Delta y^{2 (m + 1) + 1} \right)
                    \bigg]                                                    \\
                    &\leq
                    \dfrac{\BoundednessSymbol^4 |D|^{2 m + 2} C_1}{\CoercivitySymbol^2}
                    \|u^{(m + 1)}\|_{L^\infty([0, L])}^2
                    \Delta y^{4 m + 2}
                \end{align}
            \end{subequations}
            where
            \begin{align}
                C_1
                &=
                16 \Gamma^2 \max(\hat{C}_m, \hat{C}_{m + 1}) ((m + 3)!)^2     \\
                &=
                \COne
            \end{align}
            is a constant depending on only on \(m\), \(p\), \(n\), \(b\), and
            \(L\). Taking a square root yields the final result
            \begin{equation}
                |u^h(L - \Delta y) - u(L - \Delta y)|
                \leq
                \dfrac{\BoundednessSymbol^2}{\CoercivitySymbol}
                |D|^{m + 1}
                \sqrt{C_1}
                \|u^{(m + 1)}\|_{L^\infty([0, L])}
                \Delta y^{2 m + 1}.
            \end{equation}
        \end{proof}

        \begin{remark}
            An equivalent result also holds for the mesh vertex located at \(y =
            \Delta y\).
        \end{remark}

        \begin{remark}
            Since Theorem \ref{thm:last-interior-vertex-superconvergence} depends
            on global error estimates as well as estimates on the boundary
            cells, performing additional \(p\)-refinement on the boundary cells
            (i.e., \(p > 1\)) does not change the convergence rate. Put
            another way, \(p = 1\) is sufficient to obtain the improved
            convergence rate.
        \end{remark}

        \begin{remark}
            This result holds for any mesh vertex that is a fixed number of
            cells away from a boundary, e.g., for any integer \(j\), the vertex
            at \(L - j \Delta y\) is a factor of \(\Delta y\) more accurate than
            a vertex in the middle of the domain as long as the cells between
            \(L - j \Delta y\) and \(L\) are \(p\)-refined.
        \end{remark}

        %% TODO reword this remark: what precisely depends on m and p?
        %%
        %% \begin{remark}
        %%     This estimate could be improved for the case where \(m\) is odd and
        %%     \(p = 0\) by noting that the numerator of Equation
        %%     \eqref{eq:exp-sum-over-exp-bound} scales like \(O(\Delta y)\).
        %% \end{remark}

    \subsection{Rates of convergence for boundary derivatives}
        This subsection presents derivations of boundary derivative
        convergence rates for continuous finite element approximations of the
        one-dimensional model problem. We begin with a brief lemma showing
        that optimality of function values in an \(L^\infty\) norm implies
        optimality (i.e., losing one power of \(\Delta y\) for derivative taken)
        in the \(L^\infty\) norm:

        \begin{lemma}
            \label{lem:l-inf-derivative-optimality}
            Let \(I_1 = [0, \Delta y]\). Suppose that \(p(y)\) is a degree \(m\)
            polynomial approximation to \(u(y)\) on \(I_1\) satisfying the max
            norm estimate
            \begin{equation}
                \|u(y) - p(y)\|_{L^\infty(I_1)}
                \leq
                \LinfC
                %% Note that this full Sobolev norm is required by Wheeler's
                %% original proof; see "An Optimal L^oo Error Estimate for
                %% Galerkin Approximations to Solutions of Two-Point Boundary
                %% Value Problems", SINUM, 1973
                \|u(y)\|_{W^{m + 1, \infty}(I_1)}
                \Delta y^{m + 1}.
            \end{equation}
            Then
            \begin{equation}
                \left\|\dfrac{d^n}{dy^n}
                \left(u(y) - p(y)\right)\right\|_{L^\infty(I_1)}
                \leq
                \left[
                \hat{C}(m, n)
                \left(\LinfC + \dfrac{1}{m!} \right)
                +
                \dfrac{1}{(m - n)!}
                \right]
                \left\|u(y)\right\|_{W^{m + 1,\infty}(I_1)}
                \Delta y^{m - n + 1}
            \end{equation}
            where \(\hat{C}(m, n)\) is a constant dependent only on
            \(m\) and \(n\).
        \end{lemma}

        \begin{proof}
            To simplify notation, abbreviate \(L^\infty(I_1)\) and \(W^{m +
            1,\infty}(I_1)\) as \(L^\infty\) and \(W^{m + 1, \infty}\) since
            \(I_1\) is the only relevant interval. Let \(T(y)\) be the Taylor
            series approximation of degree \(m\) to \(u(y)\) centered at \(0\).
            Note that the derivative of the Taylor series is equal to the Taylor
            series of the derivative. Application of the standard Taylor
            series remainder formula on the derivative of \(u(y) - T(y)\)
            yields
            \begin{subequations}
                \begin{align}
                    \left\|\dfrac{d^n}{dy^n} \left(u(y)
                    - T(y)\right)\right\|_{L^\infty}
                    &=
                    \left\|
                    \dfrac{1}{(m - n)!}\int_0^y (y - t)^{(m - n)}
                    u^{(m + 1)}(t) dt
                    \right\|_{L^\infty}                                       \\
                    &\leq
                    \dfrac{1}{(m - n)!}
                    \left\|
                    \int_0^y \Delta y^{(m - n)} dt
                    \right\|_{L^\infty}
                    \left\|u^{(m + 1)}(y)\right\|_{L^\infty}                  \\
                    &\leq
                    \dfrac{1}{(m - n)!}
                    \left\|u(y)\right\|_{W^{m + 1, \infty}}
                    \Delta y^{m - n + 1}.
                \end{align}
            \end{subequations}
            As \(T(y)\) is a degree \(m\) polynomial, applying an inverse
            estimate yields
            \begin{subequations}
                \begin{align}
                    \left\|\dfrac{d^n}{dy^n}
                    \left(T(y) - p(y)\right)\right\|_{L^\infty}
                    &\leq \hat{C}(m, n) \Delta y^{-n}
                    \left\|T(y) - p(y)\right\|_{L^\infty}                     \\
                    &=
                    \hat{C}(m, n) \Delta y^{-n}
                    \left\|
                    \left(
                    u(y) + \int_0^y (y - t)^{m} u^{(m + 1)}(t)dt
                    \right)
                    - p(y)\right\|_{L^\infty}                                 \\
                    &\leq
                    \hat{C}(m, n) \Delta y^{-n}
                    \left(
                    \left\|u(y) - p(y)\right\|_{L^\infty}
                    +
                    \left\|
                    \dfrac{1}{m!}
                    \int_0^y (y - t)^{m} u^{(m + 1)}(t)dt
                    \right\|_{L^\infty}
                    \right)                                                   \\
                    &\leq
                    \hat{C}(m, n) \Delta y^{-n}
                    \left(
                    \LinfC \left\|u^{(m + 1)}(y)\right\|
                    _{L^\infty} \Delta y^{m + 1}
                    +
                    \dfrac{1}{m!}
                    \left\| u^{(m + 1)}(y) \right\|_{L^\infty}
                    \Delta y^{m + 1}
                    \right)                                                   \\
                    &\leq
                    \hat{C}(m, n)
                    \left(\LinfC + \dfrac{1}{m!} \right)
                    \left\|u(y) \right\|_{W^{m + 1, \infty}}
                    \Delta y^{m - n + 1}
                \end{align}
            \end{subequations}
            where \(\hat{C}(m, n)\) comes from an inverse estimate
            (see Section 4.5 of \cite{BrennerScottBook}) that depends on norm
            equivalency in finite dimensional spaces and does not depend on
            \(u(y)\) or \(\Delta y\). Applying the triangle inequality yields
            the stated result.
        \end{proof}

        \begin{theorem}
            \label{thm:cdr-1d-theorem}
            Let \(V^h \subset H^1_0([0, L])\) be the space of continuous
            piecewise polynomials defined over \(N\) cells of uniform width
            \(\Delta y\) that partition \([0, L]\) of degree \(m\) on all
            interior cells and degree \(m + p\) on all boundary cells. Let \(u^h
            \in V^h\) be the solution to Equation \eqref{eq:cdr-complex-problem}
            where \(X = V^h\) and let \(u\) be the solution to Equation
            \eqref{eq:cdr-complex-problem} with \(X = W^{m + p + 1, \infty}([0,
            L])\). Define \(D\) by Equation \eqref{eq:cdr-D-definition}. Assume
            that \(n \leq m + p\). Then there exists a constant \(C_5\)
            dependent on \(m\), \(n\), \(p\), \(b\), and \(L\) (but independent
            of \(u\), \(D\), \(\gamma\), \(\alpha\), and \(\Delta y\)) such that
            \begin{align}
                \left\|
                \dfrac{d^n}{dy^n}\left(u - u^h\right)
                \right\|_{\lastCellLinfy}
                \nonumber
                &\leq
                C_5
                \bigg[
                |D|^{2 m + p + 3} \max\left(2, \dfrac{2}{|D|\Delta y}\right)
                \dfrac{\gamma^4}{\alpha^2}
                \|u^{(m + 1)}\|_{L^\infty([0,L])} \Delta y^{2 m + 1}          \\
                &\phantom{\leq C_5 \bigg\lbrack}
                + |D|^2 \dfrac{\gamma^2}{\alpha}
                \|u(y)\|_{\lastCellWinfx{m + p + 1}}
                \Delta y^{m + p - n + 1}
                \bigg]
            \end{align}
        \end{theorem}

        \begin{proof}
            \label{thm:cdr-1d-proof}
            Notate the error at the last interior mesh vertex as
            \begin{equation}
                e_1 = u^h(L - \Delta y) - u(L - \Delta y).
            \end{equation}
            Consider an auxiliary equation (still of the form of Equation
            \eqref{eq:cdr-equation-oned}) defined on the interval \([L - \Delta
            y, L]\):
            \begin{equation}
                \label{eq:cdr-auxiliary-last-cell-bvp}
                \begin{aligned}
                    -w_{yy} + b w_y + \tilde{c} w &= \tilde{f}                \\
                    w(L - \Delta y)               &= u^h(L - \Delta y)        \\
                    w(L)                          &= u_L.
                \end{aligned}
            \end{equation}
            Consider the difference
            \begin{equation}
                e(y) = w(y) - u(y)
            \end{equation}
            between the solution of Equation
            \eqref{eq:cdr-auxiliary-last-cell-bvp} and the solution of the
            original BVP restricted to \([L - \Delta y, L]\). By definition,
            \(e(y)\) satisfies the homogeneous BVP
            \begin{equation}
                \begin{aligned}
                    -e_{yy} + b e_y + \tilde{c} e &= 0                        \\
                    e(L - \Delta y)               &= e_1                      \\
                    e(L)                          &= 0.
                \end{aligned}
            \end{equation}
            Since \(e(y)\) solves a homogeneous second-order
            constant-coefficient boundary value problem it has a simple
            closed-form solution:
            \begin{subequations}
                \begin{align}
                    e(y)
                    &=
                    \dfrac{e_1 \exp\left(\half b(\Delta y - L)\right)}
                    {2}
                    \left[
                    \dfrac
                    {\exp\left(D (L - y) + \half b y\right)}
                    {\sinh(D \Delta y)}
                    -
                    \dfrac
                    {\exp\left(D (y - L) + \half b y\right)}
                    {\sinh(D \Delta y)}
                    \right]                                                   \\
                    &=
                    \dfrac{e_1 \exp\left(\half b(\Delta y - L)\right)}
                    {2}
                    \left[
                    A(y) - B(y)
                    \right].
                \end{align}
            \end{subequations}
            Hence
            \begin{equation}
                \dfrac{d^n}{dy^n} e(y)
                =
                \dfrac{e_1 \exp\left(\half b(\Delta y - L)\right)}
                {2}
                \left[
                \left(\half b - D\right)^n A(y)
                -
                \left(\half b + D\right)^n B(y)
                \right].
            \end{equation}
            As (see Appendix \ref{sec:appendix-exp-sinh-ratio})
            \begin{equation}
                \left|\dfrac{\exp(z)}{\sinh(z)}\right|
                \leq 2 + \dfrac{2}{|z|}
            \end{equation}
            setting \(z = D \Delta y\) yields \todo{double-check the \(z\)
            substitution}
            \begin{equation}
                \label{eq:bound-exp-sinh-by-rational}
                \left|
                \dfrac{\exp(D(y - L))}{\sinh(D \Delta y)}
                \right|
                \leq
                \left|
                \dfrac
                {\exp(D \Delta y)}
                {\sinh(D \Delta y)}
                \right|
                \leq
                2 + \dfrac{2}{|D|\Delta y}.
            \end{equation}
            Since \(e(y)\) is only defined for \(L - \Delta y \leq y \leq L\),
            \(|A(y)|\) is bounded by
            \begin{equation}
                |A(y)|
                \leq
                \left(2 + \dfrac{2}{|D|\Delta y}\right)
                 \exp\left(\half b L\right).
            \end{equation}
            Similarly, as \(|\exp(D (y - L))| \leq |\exp(D (L - y))|\), \(|B(y)|
            \leq |A(y)|\). Hence
            \begin{subequations}
                \begin{align}
                    \left|
                    \dfrac{d^n}{dy^n} e(y)
                    \right|
                    &=
                    \left|
                    \dfrac{e_1 \exp\left(\half b(\Delta y - L)\right)}
                    {2}
                    \left[
                    \left(\half b - D\right)^n A(y)
                    -
                    \left(\half b + D\right)^n B(y)
                    \right]
                    \right|                                                   \\
                    &\leq
                    \left|e_1\right|
                    \exp\left(\half b (\Delta y - L)\right)
                    2
                    \left(\dfrac{1}{2}|b| + |D|\right)^n
                    \left(2 + \dfrac{2}{|D|\Delta y}\right)
                    \exp\left(\half b L\right)                                \\
                    &=
                    \left|e_1\right|
                    \exp\left(\half b \Delta y\right)
                    2
                    \left(\dfrac{1}{2}|b| + |D|\right)^n
                    \left(2 + \dfrac{2}{|D|\Delta y}\right)                   \\
                    &\leq
                    \label{eq:derivative-e-bound}
                    \eDerivativeBound{n}
                \end{align}
            \end{subequations}
            since \(|D| \geq \half |b|\).

            Let \(V^h_1 \subset V^h\) be the space of degree \(m + p\)
            polynomials with support restricted to the rightmost cell (that is,
            the cell with extent \([L - \Delta y, L]\)). Notate this cell as
            \(I_N\). Let \(w^h \in V^h_1\) be the finite element solution to the
            associated weak form of \(\eqref{eq:cdr-auxiliary-last-cell-bvp}\).
            Since \(w^h\) is a finite element discretization whose boundary
            conditions match vertex (and boundary) values of \(u^h\), \(w^h =
            u^h\) on the rightmost cell. Consider the \(L^\infty\) norm of the
            error restricted to the rightmost cell: application of Equation
            \eqref{eq:derivative-e-bound} and Lemma
            \eqref{lem:l-inf-derivative-optimality} yields
            \begin{subequations}
                \begin{align}
                    \left\|\dfrac{d^n}{dy^n}(u - u^h)\right\|_{\lastCellLinfy}
                    &=
                    \left\|
                    \dfrac{d^n}{dy^n}\left(u - w + w - u^h\right)
                    \right\|_{\lastCellLinfy}                                 \\
                    &\leq
                    \left\|\dfrac{d^n}{dy^n}e(y)\right\|_{\lastCellLinfy}
                    +
                    \left\|\dfrac{d^n}{dy^n}(w - w^h)\right\|_{\lastCellLinfy}\\
                    &\leq
                    \eDerivativeBound{n}                                      \\
                    \nonumber
                    &\hphantom{\leq} +
                    \hat{C}(m + p, n)
                    \left(\LinfC + \dfrac{1}{(m + p)!}\right)
                    \left\|w(y)\right\|_{\lastCellWinfx{m + p + 1}}
                    \Delta y^{m + p - n + 1}.
                \end{align}
            \end{subequations}
            The norm of \(w(y)\) may be bounded by using the definition of
            \(e(y)\):
            \begin{align}
                \left\|w(y)\right\|_{\lastCellWinfx{m + p + 1}}
                &\leq
                \left\|u(y)\right\|_{\lastCellWinfx{m + p + 1}}
                +
                \left\|e(y)\right\|_{\lastCellWinfx{m + p + 1}}               \\
                &\leq
                \nonumber
                \left\|u(y)\right\|_{\lastCellWinfx{m + p + 1}}               \\
                &\hphantom{\leq}
                + (m + p + 1) \eDerivativeBound{m + p}.
            \end{align}
            Let
            \begin{equation}
                C_2 = \CTwo.
            \end{equation}
            Combining terms yields
            \begin{subequations}
                \begin{align}
                    \left\|
                    \dfrac{d^n}{dy^n}\left(u - u^h\right)
                    \right\|_{\lastCellLinfy}
                    \nonumber
                    &\leq
                    \bigg(
                    \left[
                    2^{n + 2} |D|^n
                    + C_2 2^{m + p + 2}
                    |D|^{m + p} \Delta y^{m + p - n + 1}
                    \right]                                                   \\
                    &\phantom{\leq\bigg\lbrack}
                    \exp\left(\half |b|\right)
                    \max\left(2, \dfrac{2}{|D|\Delta y}\right)
                    |e_1|
                    \bigg)
                    + C_2 \|u(y)\|_{\lastCellWinfx{m + p + 1}}
                    \Delta y^{m + p - n + 1}                                  \\
                    &\leq
                    \nonumber
                    (1 + C_2) 2^{m + p + 2}
                    |D|^{m + p}
                    \exp\left(\half |b|\right)
                    \max\left(2, \dfrac{2}{|D|\Delta y}\right)
                    |e_1|                                                     \\
                    \label{eq:cdr-in-terms-of-e1}
                    &\phantom{\leq}
                    + C_2 \|u(y)\|_{W^{m + p + 1, \infty}(I_N)}
                    \Delta y^{m + p - n + 1}                                  \\
                    %% factor out the C_2
                    \nonumber
                    &\leq
                    2 C_2
                    \bigg[
                    2^{m + p + 2}
                    |D|^{m + p}
                    \exp\left(\half |b|\right)
                    \max\left(2, \dfrac{2}{|D|\Delta y}\right)
                    |e_1|                                                     \\
                    &\hphantom{\leq C_2 \bigg\lbrack}
                    + \|u(y)\|_{\lastCellWinfx{m + p + 1}}
                    \Delta y^{m + p - n + 1}\bigg].
                \end{align}
            \end{subequations}
            Since \(n \leq m + p\). Substituting in the bound for \(|e_1|\) from
            Theorem \ref{thm:last-interior-vertex-superconvergence} into
            \eqref{eq:cdr-in-terms-of-e1} yields
            \begin{subequations}
                \begin{align}
                    \left\|
                    \dfrac{d^n}{dy^n}\left(u - u^h\right)
                    \right\|_{\lastCellLinfy}
                    %% substitute in definition of |e_1|
                    \nonumber
                    &\leq
                    2 C_2 \bigg[2^{m + p + 2}
                    |D|^{m + p}
                    \exp\left(\half |b|\right)
                    \max\left(2, \dfrac{2}{|D|\Delta y}\right)                \\
                    %% |e_1| is here
                    \nonumber
                    &\hphantom{\leq 2 C_2 \bigg\lbrack}
                    \left(
                    \dfrac{\gamma^2}{\alpha} |D|^{m + 1} \sqrt{C_1}
                    \|u^{(m + 1)}\|_{L^\infty} \Delta y^{2 m + 1}
                    \right)                                                   \\
                    &\phantom{\leq 2 C_2 \bigg\lbrack}
                    + \|u(y)\|_{\lastCellWinfx{m + p + 1}}
                    \Delta y^{m + p - n + 1}\bigg]                            \\
                    %% factor stuff
                    \nonumber
                    &=
                    2 C_2 \bigg[
                    2^{m + p + 2}
                    \sqrt{C_1}
                    \exp\left(\half |b|\right)
                    |D|^{2 m + p + 1}
                    \max\left(2, \dfrac{2}{|D|\Delta y}\right)                \\
                    \label{eq:cdr-linf-error-before-c4}
                    &\hphantom{\leq 2 C_2 \bigg\lbrack}
                    \dfrac{\gamma^2}{\alpha}
                    \|u^{(m + 1)}\|_{L^\infty([0,L])} \Delta y^{2 m + 1}
                    + \|u(y)\|_{\lastCellWinfx{m + p + 1}}
                    \Delta y^{m + p - n + 1}\bigg].
                \end{align}
            \end{subequations}
            \cite{WheelerLinf} provides the bound
            \begin{equation}
                \LinfC
                \leq
                C_3(m, [0, L]) D^2 \dfrac{\gamma^2}{\alpha}
            \end{equation}
            where \(C_3\) is dependent only on the polynomial degree \(m\) and
            the domain. Hence
            \begin{subequations}
                \begin{align}
                    C_2
                    &= \CTwo                                                  \\
                    &\leq (1 + m + p) (1 + \hat{C}(m + p, n))
                    \left(1 + \LinfC\right)                                   \\
                    &\leq (1 + m + p) (1 + \hat{C}(m + p, n))
                    \left(1 + C_3(m, [0, L]) D^2 \dfrac{\gamma^2}{\alpha} \right).
                \end{align}
            \end{subequations}
            Let
            \begin{equation}
                C_4(m, p, n, b) = \CFour.
            \end{equation}
            Combining Equation \eqref{eq:cdr-linf-error-before-c4}, the
            definition of \(C_4\), and the lower bound of \(1\) on \(1/\alpha,
            \gamma\), and \(D\) yields
            \begin{subequations}
                \begin{align}
                    \left\|
                    \dfrac{d^n}{dy^n}\left(u - u^h\right)
                    \right\|_{\lastCellLinfy}
                    \nonumber
                    &\leq
                    C_4
                    \bigg[
                    \left(1 + C_3(m, [0, L]) D^2 \dfrac{\gamma^2}{\alpha}\right)
                    |D|^{2 m + p + 1}                                         \\
                    \nonumber
                    &\phantom{\leq C_4 \bigg\lbrack}
                    \max\left(2, \dfrac{2}{|D|\Delta y}\right)
                    \dfrac{\gamma^2}{\alpha}
                    \|u^{(m + 1)}\|_{L^\infty([0,L])} \Delta y^{2 m + 1}      \\
                    &\phantom{\leq C_4 \bigg\lbrack}
                    +
                    \left(
                    1 + C_3(m, [0, L]) D^2 \dfrac{\gamma^2}{\alpha}
                    \right)
                    \|u(y)\|_{\lastCellWinfx{m + p + 1}}
                    \Delta y^{m + p - n + 1}
                    \bigg]                                                    \\
                    \nonumber
                    &\leq
                    C_4 (1 + C_3(m, [0, L]))
                    \bigg[
                    |D|^{2 m + p + 3}
                    \dfrac{\gamma^4}{\alpha^2}
                    \max\left(2, \dfrac{2}{|D|\Delta y}\right)
                    \|u^{(m + 1)}\|_{L^\infty([0,L])} \Delta y^{2 m + 1}      \\
                    &\phantom{\leq C_4 (1 + C_3(m, [0, L])) \bigg\lbrack}
                    +
                    |D|^2
                    \dfrac{\gamma^2}{\alpha}
                    \|u(y)\|_{\lastCellWinfx{m + p + 1}}
                    \Delta y^{m + p - n + 1}
                    \bigg].
                \end{align}
            \end{subequations}
            Let
            \begin{equation}
                C_5 = C_5(b, m, p, n, L) = C_4 (1 + C_3(m, [0, L]))
            \end{equation}
            so \(C_5\) is dependent on \(b\), \(L\), \(m\), \(p\), and \(n\) but
            independent of \(u\) and \(D\). Hence
            \begin{align}
                \left\|
                \dfrac{d^n}{dy^n}\left(u - u^h\right)
                \right\|_{\lastCellLinfy}
                \nonumber
                &\leq
                C_5
                \bigg[
                |D|^{2 m + p + 3} \max\left(2, \dfrac{2}{|D|\Delta y}\right)
                \dfrac{\gamma^4}{\alpha^2}
                \|u^{(m + 1)}\|_{L^\infty([0,L])} \Delta y^{2 m + 1}          \\
                &\phantom{\leq C_5 \bigg\lbrack}
                + |D|^2 \dfrac{\gamma^2}{\alpha}
                \|u(y)\|_{\lastCellWinfx{m + p + 1}}
                \Delta y^{m + p - n + 1}
                \bigg]
            \end{align}
            which is the stated result.
        \end{proof}

        \begin{remark}
            An equivalent result also holds for the leftmost cell, i.e., the
            cell with extent \([0, \Delta y]\).
        \end{remark}

        \begin{remark}
            \label{remark:physical-values}
            If the constants \(b, \tilde{c}, D\) are all \(O(1)\) then the
            derivative error bound may be conveniently written in terms of a
            single scaling constant \(C\) that varies continuously with \(b\)
            and \(\tilde{c}\) (and depends on \(m\), \(p\), and \(n\)) as
            \begin{equation}
                \left\|\dfrac{d^n}{dy^n}(u - u^h)\right\|_{\lastCellLinfy}
                \leq
                C
                \left(
                \|u^{(m + 1)}(y)\|_{L^\infty([0,L])}
                \Delta y^{2 m}
                +
                \|u(y)\|_{\lastCellWinfx{m + p + 1}} \Delta y^{m + p - n + 1}
                \right).
            \end{equation}
        \end{remark}

        \begin{remark}
            \label{remark:two-asymptotic-rates-1}
            \(e(y)\) represents the contribution to the error in the last cell
            due to coupling the finite element approximation defined on the last
            cell to the rest of the computational domain. Therefore there are
            really two convergence regimes for this problem: one regime where
            the coupling error dominates and another where the local error
            dominates.
        \end{remark}

\section{Extensions to higher dimensions}
    \label{sec:extension-to-2d}
    \subsection{Overview}
        \label{subsec:overview-2d}
        Theorem \ref{thm:cdr-1d-theorem} can be generalized to higher dimensions
        when both the finite element space and mesh have a tensor product
        structure. Theorem \ref{thm:cdr-2d} does this by using the periodic
        structure in \(x\) to decouple the discretization into a sum of
        one-dimensional problems (in \(y\)) with Dirichlet boundary conditions.
        Unlike the one-dimensional analysis performed in Section
        \eqref{sec:oned-analysis}, we only consider the case where \(c \in
        \mathbb{R}\) and \(c > 0\). Since part of our analysis involves
        complex-valued test and trial functions we consider the complex weak
        problem. Note that, since the discrete test functions are real-valued,
        we still use the standard real-valued mass and stiffness matrices in the
        analysis below. The 2D model problem is as follows: Let \(V^x\) be the
        subspace of \(H^1([0, 1])\) of complex-valued functions periodic over
        \([0, 1]\). Let \(V^y = H_0^1([0, L])\). Consider the sesquilinear form
        \begin{equation}
            \label{eq:cdr-2d-sesquilinear}
            a_2(\phi, \psi)
            =
            \int_0^L\int_0^1
            \nabla \phi \cdot \nabla \bar{\psi}
            +
            \vec{b} \cdot \nabla \phi \bar{\psi}
            +
            c \phi \bar{\psi} dx dy
        \end{equation}
        and skew linear form
        \begin{equation}
            \label{eq:cdr-2d-skew}
            l_2(\psi)
            =
            \int_0^L\int_0^1
            f \bar{\psi} dx dy.
        \end{equation}
        The weak problem is, for a Hilbert space \(X \subseteq V^x \otimes V^y\),
        finding \(z \in X\) such that, for all \(\psi \in X\)
        \begin{equation}
            \label{eq:cdr-2d-weak-problem}
            a_2(z, \psi) = l_2(\psi).
        \end{equation}

        The decomposition used below (i.e., decomposition by using the tensor
        product structure) is similar to the approach used in
        \cite{DouglasLinfTensor}. Consider the tensor product finite element
        space
        \begin{equation}
            \label{eq:enriched-fe-space}
            W^h = V^{\Delta x} \otimes V^{\Delta y}
        \end{equation}
        where \(V^{\Delta x} \subset V^x\) is the space of periodic, piecewise
        linear functions defined on a uniform partition of \([0, 1]\) with
        cell width \(\Delta x\) and \(V^{\Delta y} \subset V^y\) is the space of
        piecewise polynomials defined on a uniform partition of \([0, L]\) with
        cell width \(\Delta y\), where all boundary cells (i.e., cells adjacent
        to the \(y = 0\) or \(y = L\) boundaries) have degree \(1 + p\)
        polynomials and all other cells have degree \(1\) polynomials. Note
        that \(W^h \subset V^x \otimes V^y\). As in \cite{DouglasLinfTensor},
        let
        \begin{equation}
            \label{eq:tensor-vertices}
            (\delta_i, \delta_j)
        \end{equation}
        be the coordinates of cell corners (i.e., the mesh vertices in 2D).

    \subsection{Decoupling of the 2D Problem}
        This subsection presents results showing that, with periodic boundary
        conditions and a tensor product discretization, the discrete 2D problem
        is equal to a sum of uncoupled 1D problems. Part of this analysis
        involves using test functions that are piecewise linear interpolants of
        Fourier modes. For the Hilbert space \(V^{\Delta x}\) defined in
        Equation \eqref{eq:enriched-fe-space}, define the \(k\)th Fourier mode
        interpolant \(F_k(x) \in V^{\Delta x}\) as
        \begin{equation}
            \label{eq:fourier-interpolant-def}
            F_k(x) = \Pi \exp(2 \pi I k x)
        \end{equation}
        where \(\Pi\) is the nodal piecewise linear interpolation operator onto
        \(V^{\Delta x}\) (i.e., \(F_k(x_i) = \exp(2 \pi I k x_i)\) when \(x_i\)
        is a mesh vertex). Note that \((F_{k}, F_{k'})_{L^2} = 0\) when \(k \neq
        k'\). In addition, \(F_k\) satisfies an orthogonality property with
        \(\exp(2 \pi I k' x)\):

        \begin{lemma}
            \label{lem:fourier-interpolant-orthogonality}
            Let \(V^{\Delta x}\) be the Hilbert space defined by Equation
            \eqref{eq:enriched-fe-space} and \(F_k(x)\) be the interpolant of
            \(\exp(2 \pi I k x)\) defined by Equation
            \eqref{eq:fourier-interpolant-def}. Then \(F_k(x)\) is orthogonal to
            the Fourier mode \(\exp(2 \pi I k' x)\) except when \(k - k'\) is a
            nonzero multiple of \(N\). Furthermore, if \(k = k' + j N\) (i.e.,
            \(k - k'\) is a nonzero multiple of \(N\)) then
            \begin{equation}
                \int_0^1
                F_k(x) \exp(-2 \pi I k' x) dx
                =
                \dfrac{\sin^2(\pi k' \Delta x)}{\pi^2 {k'}^2 \Delta x^2}.
            \end{equation}
        \end{lemma}

        \begin{proof}
            %% TODO include the k = k' = 0 case somewhere
            Integration by parts yields
            \begin{subequations}
                \begin{align}
                    \int_0^1 F_k(x) \exp(-2 \pi I k' x) dx
                    &=
                    \sum_{i = 0}^{N - 1}
                    \dfrac{1}{2 \pi I k'}
                    \int_{i \Delta x}^{(i + 1)\Delta x}
                    F_k'(x) \exp(-2\pi I k' x) dx                             \\
                    &=
                    \dfrac{1}{2 \pi I k'}
                    \sum_{i = 0}^{N - 1}
                    \int_{i \Delta x}^{(i + 1)\Delta x}
                    \left(
                    \exp(2 \pi I k i \Delta x)
                    \dfrac{\exp(2 \pi I k \Delta x) - 1}{\Delta x}
                    \right)
                    \exp(-2\pi I k' x) dx                                     \\
                    &=
                    \dfrac{\exp(2 \pi I k \Delta x) - 1}{2 \pi I k' \Delta x}
                    \sum_{i = 0}^{N - 1}
                    \int_{i \Delta x}^{(i + 1)\Delta x}
                    \left(
                    \exp(2 \pi I k i \Delta x)
                    \right)
                    \exp(-2\pi I k' x) dx                                     \\
                    &=
                    \dfrac{\exp(2 \pi I k \Delta x) - 1}{2 \pi I k' \Delta x}
                    \left(
                    \dfrac{I (\exp(-2 \pi I k' \Delta x) - 1)}
                    {2 \pi k'}
                    \right)
                    \sum_{i = 0}^{N - 1}
                    \exp(2 \pi I (k - k') i \Delta x)                         \\
                    &=
                    \dfrac{
                    \left(
                    \exp(2 \pi I k \Delta x) - 1
                    \right)
                    \left(
                    \exp(-2 \pi I k' \Delta x) - 1
                    \right)
                    }
                    {4 \pi^2 {k'}^2 \Delta x}
                    \sum_{i = 0}^{N - 1}
                    \exp(2 \pi I (k - k') i \Delta x)
                \end{align}
            \end{subequations}
            %% TODO this assumes k and k' are nonzero
            The summation is nonzero only if \(k - k' = j N \neq 0\) for some
            integer \(j\). Hence \(F_k(x)\) and \(\exp(2 \pi I k' x)\) are
            orthogonal except when \(k - k'\) is a nonzero multiple of \(N\).
            Finally, suppose that \(k = k' + j N\). then
            \begin{subequations}
                \begin{align}
                    \dfrac
                    {
                    \left(
                    \exp(2 \pi I (k' + j N) \Delta x) - 1
                    \right)
                    \left(
                    \exp(-2 \pi I k' \Delta x) - 1
                    \right)
                    }
                    {4 \pi^2 {k'}^2 \Delta x}
                    \sum_{i = 0}^{N - 1}
                    \exp(2 \pi I (j N) i \Delta x)
                    &=
                    \dfrac
                    {
                    4 \sin^2(\pi k' \Delta x)
                    }
                    {4 \pi^2 {k'}^2 \Delta x}
                    \left(N\right)                                            \\
                    &=
                    \dfrac
                    {
                    \sin^2(\pi k' \Delta x)
                    }
                    {
                    \pi^2 {k'}^2 \Delta x^2
                    }
                \end{align}
            \end{subequations}
            %% TODO finish discussion of error bound
        \end{proof}

        We now present the primary result of this section.
        \begin{theorem}
            \label{thm:cdr-2d-decoupling}
            Let \(W^h = V^{\Delta x} \otimes V^{\Delta y}\) be the tensor
            product finite element space defined by
            \eqref{eq:enriched-fe-space}. Suppose that \(u^h \in W^h\) is a
            solution to \eqref{eq:cdr-2d-weak-problem} with \(X = W^h\). Then
            \begin{equation}
                \label{eq:tensor-product-fourier}
                u^h(x, y) = \sum_{k = -N/2 - 1}^{N/2} F_k(x) \hat{u}^h_k(y)
            \end{equation}
            where the \(F_k(x)\) trial functions are mutually orthogonal under
            both the standard one-dimensional \(L^2\) inner product for complex
            functions as well as the sesquilinear form given by Equation
            \eqref{eq:cdr-complex-sesquilinear}, where
            \begin{align}
                (F_j, F_k)_{L^2}
                =
                \int_0^1
                F_j \bar{F}_{k}
                dx
                &= \delta_{jk} \dfrac{\lambda_{M, k}}{\Delta x}               \\
                \vec{b} &= (b_0, b_1)                                         \\
                a(F_j, F_k)
                =
                \int_0^1
                F_{j, x} \bar{F}_{k, x}
                + b_0 F_{j, x} \bar{F}_k
                + c F_j \bar{F}_k
                dx
                &= \delta_{jk} \dfrac{\lambda_{A, k}}{\Delta x}
            \end{align}
            where \(\delta_{jk}\) is the Kronecker delta, \(a(\cdot, \cdot)\)
            is the sesquilinear form defined by Equation
            \eqref{eq:cdr-complex-sesquilinear}, \(\lambda_M\) and \(\lambda_A\)
            are scalars, and each \(\hat{u}^h_k(y)\) satisfies the
            one-dimensional boundary value problem
            \begin{equation}
                \int_0^L
                \hat{u}^h_{k, y} \bar{\phi}_y
                + b_1 \hat{u}^h_{k, y} \bar{\phi}
                + \dfrac{\lambda_{A, k}}{\lambda_{M, k}}
                \hat{u}^h_k \bar{\phi} dy
                =
                \int_0^L
                \int_0^1
                \dfrac{\Delta x}{\lambda_{M, k}}
                f(x, y) F_k(x)
                \bar{\phi} dx dy,
                \forall \phi(y) \in V^{\Delta y}.
            \end{equation}
        \end{theorem}

        \begin{proof}
            Let \(A\) and \(M\) be the stiffness and mass matrices,
            respectively, coming from discretizations of \(a(\cdot, \cdot)\) and
            \((\cdot, \cdot)_{L^2}\) with standard piecewise linear hat
            functions. Since \(V^{\Delta x}\) consists of piecewise linear
            functions and the boundary conditions in the \(x\) direction are
            periodic, both \(A\) and \(M\) are circulant matrices and therefore
            share the same set of mutually orthogonal eigenvectors, which are
            the mesh vertex values of the Fourier interpolants \(F_k(x)\)
            defined by Equation \eqref{eq:fourier-interpolant-def}. Since the
            \(i\)th row of \(A\) is
            \begin{equation}
                A_i =
                \begin{pmatrix}
                    0, & 0, & \cdots, &
                    %% lower diagonal
                    \dfrac{-1}{\Delta x} - \dfrac{b_0}{2} + \dfrac{c \Delta x}{6}, &
                    %% main diagonal
                    \dfrac{2}{\Delta x}  + \dfrac{4 c \Delta x}{6}, &
                    %% upper diagonal
                    \dfrac{-1}{\Delta x} + \dfrac{b_0}{2} + \dfrac{c \Delta x}{6},
                    \cdots, & 0, & 0
                \end{pmatrix}
            \end{equation}
            and, similarly, the \(i\)th row of \(M\) is
            \begin{equation}
                M_i =
                \begin{pmatrix}
                    0, & 0, & \cdots, &
                    %% lower diagonal
                     \dfrac{\Delta x}{6}, &
                    %% main diagonal
                    \dfrac{4 \Delta x}{6}, &
                    %% upper diagonal
                    \dfrac{\Delta x}{6}, &
                    \cdots, & 0, & 0
                \end{pmatrix}
            \end{equation}
            The classic formula for the eigenvalues of a circulant matrices
            provides the \(k\)th eigenvalues of \(M\)
            \begin{equation}
                \lambda_{M, k} =
                \dfrac{\Delta x (2 \cos(2 \pi \Delta x k) + 4)}{6}
            \end{equation}
            and \(A\)
            \begin{equation}
                \lambda_{A, k} =
                \dfrac
                {2 c {\Delta x}^{2}
                + 3 I b_{0} {\Delta x} \sin\left(2 \pi {\Delta x} k\right)
                + {\left(c {\Delta x}^{2} - 6\right)} \cos\left(2 \pi {\Delta x} k\right)
                + 6
                }
                {3 {\Delta x}}.
            \end{equation}
            Note that \(\lambda_{M,k} > 0\). Since \(A\) and \(M\) have complete
            sets of eigenvectors, the set \(\left\{F_k(x)\right\}\) is a basis
            for \(V^{\Delta x}\). Therefore, expressing \(u^h\) in this new
            basis (instead of the usual hat functions) yields
            \begin{equation}
                u^h = \sum_{j,m} c_{jm} F_m(x) Y_j(y),
                F_m(x) \in V^{\Delta x},
                Y_j(y) \in V^{\Delta y}.
            \end{equation}
            By the same argument, consider a test function \(\varphi = F_k(x)
            Y_l(y) \in W^h\). Plugging \(u^h\) and \(\varphi\) into the finite
            element problem given by Equation \eqref{eq:cdr-2d-weak-problem}
            with \(X = W^h\) yields
            \begin{subequations}
                \begin{align}
                    \int_0^L \int_0^1 u^h_x \bar{\varphi}_x + b_0 u^h_x \bar{\varphi}
                    + c u^h \bar{\varphi} dx dy
                    +
                    \int_0^L \int_0^1 u^h_y (\bar{\varphi}_y + b_1 \bar{\varphi}) dx dy
                    &=
                    \int_0^L \int_0^1 f \bar{\varphi} dx dy                   \\
                    \sum_{j, m} c_{jm}
                    \int_0^L
                    a(F_m, F_k)
                    Y_j(y) \bar{Y}_l(y)
                    +
                    (F_m, F_k)_{L^2}
                    Y_{j,y}(y) (\bar{Y}_{l,y}(y) + b_1 \bar{Y}_l(y)) dy
                    &= \int_0^L \int_0^1 f \bar{\varphi} dx dy                \\
                    \sum_j c_{jk}
                    \int_0^L \dfrac{\lambda_{A, k}}{\Delta x} Y_j(y) \bar{Y}_l(y)
                    +
                    \dfrac{\lambda_{M, k}}{\Delta x}
                    Y_{j,y}(y) (\bar{Y}_{l,y}(y) + b_1 \bar{Y}_l(y)) dy
                    &= \int_0^L \int_0^1 f \bar{\varphi} dx dy
                \end{align}
            \end{subequations}
            multiplying both sides by \(\Delta x/\lambda_{M, k}\) yields
            \begin{equation}
                \sum_j c_{jk}
                \int_0^L Y_{j,y}(y) (\bar{Y}_{l,y}(y) + b_1 \bar{Y}_l(y))
                +
                \dfrac{\lambda_{A, k}}{\lambda_{M, k}}
                Y_j(y) \bar{Y}_l(y)
                dy
                = \int_0^L \int_0^1
                \dfrac{\Delta x}{\lambda_{M, k}}
                f \bar{F}_m(x) \bar{Y}_l(y) dx dy.
            \end{equation}
            Hence, defining
            \begin{equation}
                \hat{u}^h_k(y) = \sum_j c_{jk} Y_j(y)
            \end{equation}
            yields the one-dimensional convection-diffusion-reaction problem
            \begin{equation}
                \label{eq:oned-decoupled-problem}
                \int_0^L
                \hat{u}^h_{k,y}(y) \bar{Y}_{l,y}(y)
                + b_1 \hat{u}^h_{k,y}(y) \bar{Y}_l(y)
                + \dfrac{\lambda_{A,k}}{\lambda_{M,k}}
                \hat{u}^h_k(y) \bar{Y}_l(y) dy =
                \int_0^L \hat{f}_k^h(y) \bar{Y}_l(y) dy
            \end{equation}
            for all \(Y_l(y) \in V^{\Delta y}\), where
            \begin{equation}
                \hat{f}_k^h(y)
                =
                \int_0^1
                \dfrac{\Delta x}{\lambda_{M, k}}
                f(x, y) \bar{F}_k(x)
                dx
            \end{equation}
            is the \(L^2\) projection of the \(x\)-component of \(f\) onto
            \(F_k(x)\). Therefore, by construction
            \begin{equation}
                u^h(x, y) = \sum_{k} F_k(x) \hat{u}^h_k(y)
            \end{equation}
            where each \(\hat{u}^h_k(y)\) satisfies Equation
            \eqref{eq:oned-decoupled-problem}.
        \end{proof}

        \begin{theorem}
            The resulting one-dimensional finite element problem implied by
            Equation \eqref{eq:oned-decoupled-problem} is well-posed.
        \end{theorem}

        \begin{proof}
            The eigenvalue ratio is
            \begin{equation}
                \lambda_k = \dfrac{\lambda_{A,k}}{\lambda_{M,k}}
                =
                \dfrac
                {
                2 c {\Delta x}^{2}
                + 3 I b_{0} {\Delta x} \sin\left(2 \pi {\Delta x} k\right)
                + {\left(c {\Delta x}^{2} - 6\right)}
                \cos\left(2 \pi {\Delta x} k\right)
                + 6
                }
                {
                {\Delta x}^{2}
                {\left(\cos\left(2 \pi {\Delta x} k\right) + 2\right)}
                }.
            \end{equation}
            Notate the real part of the eigenvalue ratio as
            \begin{equation}
                \real{\lambda_k}
                =
                \dfrac
                {
                2 c \Delta x^2 + (c \Delta x^2 - 6) \cos(2 \pi \Delta x k) + 6
                }
                {
                \Delta x^2 (\cos(2 \pi \Delta x k) + 2)
                }.
            \end{equation}
            Note that
            \begin{subequations}
                \begin{align}
                    \real{\lambda_k}
                    &\geq
                    \dfrac
                    {
                    2 c \Delta x^2 + (c \Delta x^2 - 6) \cos(2 \pi \Delta x k) + 6
                    }
                    {
                    3 \Delta x^2
                    }                                                         \\
                    &=
                    \dfrac
                    {
                    (2 + \cos(2 \pi \Delta x k)) c \Delta x^2
                    + 6 (1 - \cos(2 \pi \Delta x k))
                    }
                    {
                    3 \Delta x^2
                    }                                                         \\
                    \label{eq:cdr-eigenvalue-ratio-lower-bound}
                    &\geq
                    c
                \end{align}
            \end{subequations}
            and
            \begin{subequations}
                \begin{align}
                    \real{\lambda_k}
                    &\leq
                    \dfrac
                    {
                    2 c \Delta x^2 + (c \Delta x^2 - 6) \cos(2 \pi \Delta x k) + 6
                    }
                    {
                    \Delta x^2
                    }                                                         \\
                    &=
                    \dfrac
                    {
                    (2 + \cos(2 \pi \Delta x k)) c \Delta x^2
                    + 6 (1 - \cos(2 \pi \Delta x k))
                    }
                    {
                    \Delta x^2
                    }                                                         \\
                    &\leq
                    \label{eq:cdr-eigenvalue-ratio-upper-bound}
                    \dfrac
                    {
                    6 + 3 c \Delta x^2
                    }
                    {
                    \Delta x^2
                    }.
                \end{align}
            \end{subequations}
            Hence, for any \(\Delta x > 0\) \(c < \real{\lambda_k} < \infty\),
            so by Theorem \ref{thm:complex-oned-wellposedness} the resulting
            one-dimensional problem is well-posed.
        \end{proof}

        \begin{remark}
            An unusual feature of this well-posedness argument is the presence
            of the continuity constant that scales like \(O(1/\Delta x^2)\):
            this is due to the presence of two \(x\)-derivatives in the low
            order term of the \(y\) discretization.
        \end{remark}

    \subsection{Boundary derivative convergence of the 2D problem}
        \begin{theorem}
            \label{thm:cdr-2d}
            Consider the discretization of Equation \eqref{eq:cdr-equation}
            described in Equations
            \eqref{eq:cdr-2d-sesquilinear}-\eqref{eq:enriched-fe-space}. Let
            \(C_6\) and \(C_8\) be constants dependent on the coefficients of
            \eqref{eq:cdr-equation}, the domain, the forcing function (in
            particular, its regularity), and the number of derivatives \(n\).
            Assume that \(f\) is at least \(18 + p\) \todo{double check that
            this is the precise number} times differentiable in the \(y\)
            direction. This discretization recovers the 1D estimate given by
            Theorem \ref{thm:cdr-1d-theorem} at isolated points along the
            nonperiodic boundary with an additional error term coming from the
            \(x\)-discretization:
            \begin{equation}
                \label{eq:cdr-2d-convergence-recovery}
                \left|
                \dfrac{d^n}{dy^n}
                \left(u(x, y) - u^h(x, y) \right)
                \bigg|_{(\delta_i, \delta_{N^*})}
                \right|
                \leq
                C_6 \Delta x^2
                +
                C_8 (\Delta y^2 + \Delta y^{2 + p - n})
            \end{equation}
            where \(N^* = 0\) or \(N^* = N\).
        \end{theorem}

        \begin{proof}
            Since \(u(x, y)\) and \(f(x, y)\) are periodic in the \(x\)
            direction and smooth they are equal to their Fourier series:
            \begin{align}
                u(x, y)
                &=
                \sum_{k = -\infty}^\infty
                \exp(2 \pi I k x) \hat{u}_k(y)                                \\
                f(x, y)
                &=
                \sum_{k = -\infty}^\infty
                \exp(2 \pi I k x) \hat{f}_k(y).
            \end{align}
            Hence each \(\hat{u}_k\) solves the BVP
            \begin{equation}
                -\hat{u}_{k,yy}
                +
                b_1 \hat{u}_{k,y}
                +
                (4 \pi^2 k^2 + 2 \pi I b_0 k + c) \hat{u}_k
                = \hat{f}_k
            \end{equation}
            corresponding to the weak problem
            \begin{equation}
                \label{eq:fourier-transform-1d-problem}
                \int_0^L \hat{u}_{k,y} \bar{\phi}_y
                + b_1 \hat{u}_{k,y} \bar{\phi}
                + (4 \pi^2 k^2 + 2 \pi I b_0 k + c) \hat{u}_k \bar{\phi}
                dy
                =
                \int_0^L
                \hat{f}_k \bar{\phi} dy.
            \end{equation}
            By Equation \eqref{eq:tensor-product-fourier}, \(u^h\) may be
            written as
            \begin{equation}
                u^h(x, y) =
                \sum_{k = -\infty}^\infty
                F_k(x) \hat{u}^h_k(y)
            \end{equation}
            where, for \(|k| > 1/(2 \Delta x)\), \(F_k(x) = \hat{u}^h_k(y) =
            0\). Hence, assume that \(|k| \leq 1/(2 \Delta x)\).
            \todo{Expand on this assumption.} The difference between the
            eigenvalue ratio and the low order coefficient in Equation
            \eqref{eq:fourier-transform-1d-problem} is equal to, by a
            Taylor series expansion in \(\Delta x\),
            \begin{subequations}
                \begin{align}
                    \dfrac{\lambda_{A,k}}{\lambda_{M,k}}
                    &=
                    4 \pi^{2} k^{2} + 2 I \pi b_{0} k + c
                    + \frac{4}{3} \pi^{4} k^{4} {\Delta x}^{2}
                    + \frac{1}{45} \,
                    {\left(8 \, \pi^{6} k^{6} - 8 I \pi^{5} b_{0} k^{5}\right)}
                    {\Delta x}^{4}
                    + \cdots                                                  \\
                    &=
                    4 \pi^{2} k^{2} + 2 I \pi b_{0} k + c
                    + R.
                \end{align}
            \end{subequations}
            Since \(|k| \leq 1/(2 \Delta x)\), converting all terms in the
            Taylor series of \(R\) to constant multiples of \(k^4 \Delta x^2\)
            or \(k^4 \Delta x^3\) yields
            \begin{equation}
                \label{eq:R-bound-from-Taylor}
                |R| \leq C_R(b_0) k^4 \Delta x^2
            \end{equation}
            where \(C_R(b_0)\) is a constant dependent on \(b_0\). Let
            \(\hat{v}_k(y) \in V^y\) be the solution to the semidiscretization
            in \(x\) of Equation \eqref{eq:cdr-2d-sesquilinear}, i.e., the
            solution to the weak problem
            \begin{equation}
                \label{eq:cdr-2d-semidiscretization}
                \int_0^L \hat{v}_{k,y} \bar{\phi}_y
                + b_1 \hat{v}_{k,y} \bar{\phi}
                + (4 \pi^2 k^2 + 2 \pi I b_0 k + c + R) \hat{v}_k \bar{\phi}
                dy
                =
                \int_0^L
                \hat{f}^h_k(y) \bar{\phi} dy, \forall \phi \in V^y
            \end{equation}
            where \(\hat{v}_k(y) = 0\) for \(|k| > 1/(2 \Delta x)\). The rest of
            the proof follows from a triangle inequality argument involving
            \(\hat{u}_k\), \(\hat{v}_k\), and \(\hat{u}^h_k\). Consider the
            decomposition
            \begin{subequations}
                \begin{align}
                    u(\delta_i, y) - u^h(\delta_i, y)
                    &=
                    \sum_{k = -\infty}^\infty
                    \left(
                    \exp(2 \pi I k \delta_i) \hat{u}_k(y)
                    -
                    F_k(\delta_i) \hat{u}^h_k(y)
                    \right)                                                   \\
                    &=
                    \sum_{k = -\infty}^\infty
                    \left(
                    \exp(2 \pi I k \delta_i) \hat{u}_k(y)
                    - F_k(\delta_i) \hat{v}_k(y)
                    + F_k(\delta_i) \hat{v}_k(y)
                    -
                    F_k(\delta_i) \hat{u}^h_k(y)
                    \right)                                                   \\
                    \nonumber
                    &=
                    \sum_{k = -\infty}^\infty
                    \bigg(
                    \exp(2 \pi I k \delta_i)
                    \left(
                    \hat{u}_k(y)
                    - \hat{v}_k(y)
                    \right)
                    +
                    (\exp(2 \pi I k \delta_i) - F_k(\delta_i))
                    \hat{v}_k(y)                                              \\
                    &\hphantom{= \sum_{k = -\infty}^\infty \bigg\lbrack}
                    +
                    F_k(\delta_i)
                    \left(
                    \hat{v}_k(y) - \hat{u}^h_k(y)
                    \right)
                    \bigg)                                                    \\
                    &=
                    \sum_{k = -\infty}^\infty
                    \left(
                    \underbrace{
                    \exp(2 \pi I k \delta_i)
                    \left(
                    \hat{u}_k(y)
                    - \hat{v}_k(y)
                    \right)}_{=(1)}
                    +
                    \underbrace{F_k(\delta_i)
                    \left(
                    \hat{v}_k(y) - \hat{u}^h_k(y)
                    \right)}_{=(2)}
                    \right)
                \end{align}
            \end{subequations}
            since \(\exp(2 \pi I k \delta_i) = F_k(\delta_i)\) by the definition
            of \(F_k(x)\) in Equation \eqref{eq:fourier-interpolant-def}.

            \paragraph{Bounding \texorpdfstring{\((1)\)}{1}}
            Let \(\hat{e}_k = \hat{v}_k - \hat{u}_k\). Then \(\hat{e}_k\)
            satisfies the weak boundary value problem
            \begin{equation}
                \int_0^L \hat{e}_{k,y} \bar{\phi}_y
                + b_1 \hat{e}_{k,y} \phi
                + (4 \pi^2 k^2 + 2 \pi I b_0 k + c) \hat{e}_k \bar{\phi}
                dy
                =
                \int_0^L
                \left(
                \left(
                \hat{f}^h_k
                - \hat{f}_k
                \right)
                - R \hat{v}_k
                \right)
                \bar{\phi} dy.
            \end{equation}
            \(\hat{e}_k\) can be bounded by standard Sobolev estimates since the
            right-hand side (a combination of the error in the approximate
            Fourier mode and error in the low order term) is relatively small
            (\(O(\Delta x^2)\) for small \(k\)). Since \(\hat{e}_k\) solves an
            elliptic problem with smooth data, a regularity estimate yields
            (see \cite{EvansBook}, section 6.3.2) \todo{double check the scaling
            of the constant w.r.t. the data; in that chapter we only see
            \(\|c\|_{L^\infty}\) but this is not explicitly shown in the
            particular theorem I cite for controlling the \(H^{r + 2}\) norm}
            \begin{equation}
                \label{eq:evans-regularity-estimate}
                \|\hat{e}_k\|_{H^{r + 2}}
                \leq
                C([0, L], r)
                \left(
                1
                + |b_1|
                + \left|4 \pi^2 k^2 + 2 \pi I b_0 k + c\right|
                \right)
                \left\|
                \left(
                \left(
                \hat{f}^h_k
                - \hat{f}_k
                \right)
                - R \hat{v}_k
                \right)
                \right\|_{H^r}.
            \end{equation}
            Due to the homogeneous boundary conditions, the fundamental
            theorem of calculus implies that
            \begin{equation}
                \label{eq:linf-derivative-estimate}
                \left\|\dfrac{d^n}{dy^n} \hat{e}_k\right\|_{L^\infty}
                \leq
                L
                \left\|\dfrac{d^n}{dy^n} \hat{e}_k\right\|_{H^1}
                \leq
                L \left\|\hat{e}_k\right\|_{H^{n + 1}}.
            \end{equation}
            Combining Equations
            \eqref{eq:evans-regularity-estimate}%
            -\eqref{eq:linf-derivative-estimate}
            yields, for \(1 \leq n\),
            \begin{subequations}
                \begin{align}
                    \left\|\dfrac{d^n}{dy^n} \hat{e}_k\right\|_{L^\infty}
                    &\leq
                    C([0, L], n - 1)
                    \left(
                    1
                    + |b_1|
                    + \left|4 \pi^2 k^2 + 2 \pi I b_0 k + c\right|
                    \right)
                    \left\|
                    \left(
                    \hat{f}^h_k
                    - \hat{f}_k
                    \right)
                    - R \hat{v}_k
                    \right\|_{H^{n - 1}}                                      \\
                    \label{eq:e-k-original-linf-estimate}
                    &\leq
                    C([0, L], n - 1)
                    \left(
                    1
                    + |b_1|
                    + \left|4 \pi^2 k^2 + 2 \pi I b_0 k + c\right|
                    \right)
                    \left(
                    \left\|
                    \hat{f}^h_k
                    - \hat{f}_k
                    \right\|_{H^{n - 1}}
                    + R
                    \left\|
                    \hat{v}_k
                    \right\|_{H^{n - 1}}
                    \right).
                \end{align}
            \end{subequations}
            Lemma \ref{lem:fourier-interpolant-orthogonality} bounds the
            %% TODO cite definition of \lambda_{M,k}
            error in approximating \(\hat{f}_k\):
            \begin{subequations}
                \begin{align}
                    \hat{f}^h_k(y) - \hat{f}_k(y)
                    &=
                    \int_0^1
                    \dfrac{\Delta x}{\lambda_{M,k}}
                    f(x, y) \bar{F}_k(x) dx
                    -
                    \hat{f}_k(y)                                              \\
                    &=
                    \int_0^1
                    \dfrac{\Delta x}{\lambda_{M,k}}
                    \left(
                    \sum_{j = -\infty}^\infty
                    \hat{f}_{k + j N}(y)
                    \exp(2 \pi I (k + j N) x)
                    \right)
                    \bar{F}_k(x) dx
                    -
                    \hat{f}_k(y)                                              \\
                    \nonumber
                    &=
                    \left(
                    \int_0^1
                    \dfrac{\Delta x}{\lambda_{M,k}}
                    \exp(2 \pi I k x)
                    \bar{F}_k(x)
                    - 1 dx
                    \right) \hat{f}_k(y)                                      \\
                    &\hphantom{=} +
                    \dfrac{\Delta x}{\lambda_{M,k}}
                    \sum_{\substack{j = -\infty \\ j \neq 0}}^\infty
                    \hat{f}_{k + j N}(y)
                    \int_0^1
                    \exp(2 \pi I (k + j N) x) \bar{F}_k(x) dx                 \\
                    \label{eq:f-k-difference-with-summation}
                    &=
                    \left(
                    \dfrac{\sin^2(\pi \Delta x k)}{\pi^2 \Delta x^2 k^2}
                    \dfrac{\Delta x}{\lambda_{M,k}}
                    - 1
                    \right)
                    \hat{f}_k(y)
                    +
                    \dfrac{\Delta x}{\lambda_{M,k}}
                    \sum_{\substack{j = -\infty \\ j \neq 0}}^\infty
                    \hat{f}_{k + j N}(y)
                    \dfrac{4 N^2 \sin^2(\pi k/N)}{\pi^2 (k + j N)^2}.
                \end{align}
            \end{subequations}
            Due to the regularity assumption on \(f\), there exists a constant
            \(C(f)\) independent of \(x\), \(y\), and \(k\) such that
            \begin{align}
                \label{eq:f-k-decay}
                \left|
                \dfrac{d^n}{dy^n}
                \hat{f}_k(y)
                \right|
                &\leq C(f) k^{-q + n}                                         \\
                \label{eq:f-h-k-decay}
                \left|
                \dfrac{d^n}{dy^n}
                \hat{f}^h_k(y)
                \right|
                &\leq C(f) k^{-q + n}.
            \end{align}
            Note that
            %% TODO cite definition of eigenvalue
            \begin{equation}
                \label{eq:dx-eigenvalue-ratio-bound}
                \dfrac{\Delta x}{\lambda_{M,k}}
                =
                \dfrac{6}{2 \cos(2 \pi k \Delta x) + 4}
                \Rightarrow
                \left|
                \dfrac{\Delta x}{\lambda_{M,k}}
                \right|
                \leq 3.
            \end{equation}
            Substituting Equations \eqref{eq:f-k-decay} and
            \eqref{eq:dx-eigenvalue-ratio-bound} into the \(n\)th derivative of
            Equation \eqref{eq:f-k-difference-with-summation} yields
            \begin{subequations}
                \begin{align}
                    \left|
                    \dfrac{d^n}{dy^n}
                    \left(
                    \hat{f}^h_k(y) - \hat{f}_k(y)
                    \right)
                    \right|
                    &\leq
                    C(f)
                    \left(
                    \left|
                    \dfrac{\sin^2(\pi \Delta x k)}{\pi^2 \Delta x^2 k^2}
                    \dfrac{\Delta x}{\lambda_{M,k}}
                    - 1
                    \right|
                    k^{-q + n}
                    +
                    \dfrac{12 N^2 \sin^2(\pi k/N)}{\pi^2}
                    \sum_{\substack{j = -\infty \\ j \neq 0}}^\infty
                    \dfrac{1}{(N j + k)^{2 + q - n}}
                    \right)                                                   \\
                    &\leq
                    C(f)
                    \left(
                    \left|
                    \dfrac{\sin^2(\pi \Delta x k)}{\pi^2 \Delta x^2 k^2}
                    \dfrac{\Delta x}{\lambda_{M,k}}
                    - 1
                    \right|
                    k^{-q + n}
                    +
                    \dfrac{12 N^2}{\pi^2}
                    \sum_{\substack{j = -\infty \\ j \neq 0}}^\infty
                    \dfrac{1}{(N j + k)^{2 + q - n}}
                    \right)                                                   \\
                    \nonumber
                    &=
                    C(f)
                    \Bigg(
                    \left|
                    \dfrac{\sin^2(\pi \Delta x k)}{\pi^2 \Delta x^2 k^2}
                    \dfrac{\Delta x}{\lambda_{M,k}}
                    - 1
                    \right|
                    k^{-q + n}
                    +
                    \dfrac{12 N^2}{\pi^2}
                    \left(N^{-2 - q + n}\right)                               \\
                    \label{eq:f-k-difference-with-zeta}
                    &\phantom{= C(f)\bigg\lbrack}
                    \left[
                    \zeta_H(2 + q - n, 1 + k N^{-1})
                    +
                    \zeta_H(2 + q - n, 1 - k N^{-1})
                    \right]
                    \Bigg)
                \end{align}
            \end{subequations}
            where \(\zeta_H\) is the Hurwitz zeta function. Rearranging the
            first term in Equation \eqref{eq:f-k-difference-with-zeta} and
            performing a Taylor series expansion of the first term around \(k
            \Delta x = 0\) provides the bound
            \begin{subequations}
                \begin{align}
                    \left|
                    \dfrac{\sin^2(\pi \Delta x k)}{\pi^2 k^2 \Delta x^2}
                    \dfrac{6}{2\cos(2 \pi k \Delta x) + 4}
                    - 1
                    \right|
                    &=
                    \left|
                    \dfrac{6}{2\cos(2 \pi k \Delta x) + 4}
                    \right|
                    \left|
                    \dfrac{\sin^2(\pi \Delta x k)}{\pi^2 k^2 \Delta x^2}
                    -
                    \dfrac{2\cos(2 \pi k \Delta x) + 4}{6}
                    \right|                                                   \\
                    &\leq
                    3
                    \left|
                    \dfrac{\sin^2(\pi \Delta x k)}{\pi^2 k^2 \Delta x^2}
                    -
                    \dfrac{2\cos(2 \pi k \Delta x) + 4}{6}
                    \right|                                                   \\
                    &=
                    3
                    \left|
                    \sin^2(\pi \Delta x k)
                    -
                    \dfrac{2\cos(2 \pi k \Delta x) + 4}{6}
                    \pi^2 k^2 \Delta x^2
                    \right|
                    \dfrac{1}{\pi^2 k^2 \Delta x^2}                           \\
                    &\leq
                    3
                    %% TODO for consistent style, distribute the powers
                    \left|
                    \dfrac{1}{3}(k \pi \Delta x)^4
                    +
                    \left(
                    \dfrac{928}{3} \pi^8 (k \Delta x)^2
                    + 192 \pi^6
                    \right)
                    (k \Delta x)^6
                    \right|
                    \dfrac{1}{\pi^2 k^2 \Delta x^2}                           \\
                    &=
                    3
                    \left|
                    \dfrac{1}{3}(k \pi \Delta x)^2
                    +
                    \left(
                    \dfrac{928}{3} \pi^8 (k \Delta x)^2
                    + 192 \pi^6
                    \right)
                    (k \Delta x)^4
                    \right|                                                   \\
                    &\leq C_T k^2 \Delta x^2
                    \label{eq:f-remainder-taylor-part}
                \end{align}
            \end{subequations}
            since \(|k| \leq 1/(2 \Delta x)\), where \(C_T\) is a constant
            dependent on the coefficients of the Taylor series expansion. Since
            \begin{equation}
                \half \leq 1 - k N^{-1}
            \end{equation}
            The \(\zeta_H\) terms are bounded by
            \begin{subequations}
                \begin{align}
                    \zeta_H\left(2 + q - n, \half\right)
                    &=
                    \sum_{j = 0}^\infty \dfrac{1}{(\half + j)^{2 + q - n}}    \\
                    &\leq
                    2^{2 + q - n} + \sum_{j = 1}^\infty \dfrac{1}{j^2}        \\
                    &= 2^{2 + q - n} + 1
                \end{align}
            \end{subequations}
            due to the assumption on the regularity of \(f\). Hence
            \begin{equation}
                \label{eq:f-remainder-hurwitz-part}
                \zeta_H(2 + q - n, 1 + k N^{-1})
                +
                \zeta_H(2 + q - n, 1 - k N^{-1})
                \leq
                2^{3 + q - n} + 2.
            \end{equation}
            Hence, substituting \(N = 1/\Delta x\), Equation
            \eqref{eq:f-remainder-taylor-part}, and Equation
            \eqref{eq:f-remainder-hurwitz-part} into Equation
            \eqref{eq:f-k-difference-with-zeta} yields
            \begin{align}
                \label{eq:f-h-k-minus-f-k-derivative-bound}
                \left|
                \dfrac{d^n}{dy^n}
                \left(
                \hat{f}^h_k(y) - \hat{f}_k(y)
                \right)
                \right|
                &\leq
                C(f)
                \left[
                \left(
                C_T k^2\Delta x^2
                \right)
                k^{-q + n}
                +
                (2^{3 + q - n} + 2)
                \Delta x^{q - n}
                \right]                                                       \\
                &\leq
                C(f)
                \left[
                C_T k^{-q + n + 2} \Delta x^2
                +
                (2^{3 + q - n} + 2) \Delta x^{q - n}
                \right].
            \end{align}
            Since \(\hat{v}_k\) is a weak solution to Equation
            \eqref{eq:cdr-2d-semidiscretization}, it satisfies the derivative
            estimate
            \begin{subequations}
                \begin{align}
                    \|\hat{v}_k\|_{H^{r + 2}}
                    &\leq
                    C([0, L], r)
                    \left(
                    1 +
                    | b_1 | +
                    \left|\dfrac{\lambda_{A,k}}{\lambda_{M,k}}\right|
                    \right)
                    \|\hat{f}^h_k\|_{H^r}                                     \\
                    \label{eq:v-k-hat-regularity-estimate}
                    &\leq
                    C([0, L], r)
                    \left(
                    1 +
                    | b_1 | +
                    \left|\dfrac{\lambda_{A,k}}{\lambda_{M,k}}\right|
                    \right)
                    C(f) k^{-q + r}.
                \end{align}
            \end{subequations}
            Hence, substituting Equation \eqref{eq:R-bound-from-Taylor},
            Equation \eqref{eq:f-h-k-minus-f-k-derivative-bound}, and Equation
            \eqref{eq:v-k-hat-regularity-estimate} into Equation
            \eqref{eq:e-k-original-linf-estimate} yields
            \begin{subequations}
                \begin{align}
                    \nonumber
                    \left\|
                    \dfrac{d^n}{dy^n}
                    \hat{e}_k
                    \right\|_{L^\infty}
                    &\leq
                    C([0, L], n - 1) C(f)
                    \left(
                    1 +
                    |b_1| +
                    \left|4 \pi^2 k^2 + 2 \pi I k b_0 + c\right|
                    \right)
                    \bigg(
                    C_T k^{1 - q + n} \Delta x^2                              \\
                    &\phantom{\leq}
                    +
                    (2^{2 + q - n} + 2)
                    \Delta x^{q - n}
                    + C_R(b_0) C([0, L], n) k^4 \Delta x^2
                    \left(
                    1 +
                    |b_1| +
                    \left|\dfrac{\lambda_{A,k}}{\lambda_{M,k}}\right|
                    \right)
                    k^{-q + \max(n - 3, 0)}
                    \bigg)                                                    \\
                    \nonumber
                    &\leq
                    C_6
                    \bigg(
                    k^{4 - q + n} \Delta x^2
                    + (2^{2 + q - n} + 2) k^2 \Delta x^{q - n}
                    + k^{8 - q + \max(n - 3, 0)} \Delta x^2
                    \bigg)
                \end{align}
            \end{subequations}
            where
            \begin{equation}
                C_6 = \CSixSymbolic
            \end{equation}
            is a constant independent of \(k\) and \(\Delta x\). Due to the
            regularity assumption on \(f\) the exponent is bounded by \(10
            + \max(n - 3, 0) \leq q\): hence this sum converges and
            \begin{equation}
                \left|\sum_{k = -\infty}^\infty
                \exp(2 \pi I k \delta_i)
                \dfrac{d^n}{dy^n}
                \hat{e}_k(y)
                \right|
                \leq
                \CSixSymbolic \Delta x^2
            \end{equation}
            since, for \(|k| > 1/(2 \Delta x)\), due to the regularity
            assumption on \(f\)
            \begin{equation}
                \left|
                \dfrac{d^n}{dy^n}
                \hat{e}_k(y)
                \right|
                =
                \left|
                \dfrac{d^n}{dy^n}
                \hat{u}_k(y)
                \right|
                \leq C k^{-18 + p}
                \leq C k^{-16 + p} \Delta x^2
            \end{equation}
            for some constant \(C\) dependent on \(C(f)\), so the contribution
            from the unresolved modes is also bounded by a constant multiple of
            \(\Delta x^2\).

            \paragraph{Bounding \texorpdfstring{\((2)\)}{2}}
            Consider the error in the full discretization relative to
            the semidiscretization (i.e., \(\hat{u}_k^h - \hat{v}_k\)). The
            finite element solution \(\hat{u}_k^h\) satisfies the boundary value
            problem
            \begin{equation}
                \int_0^L
                \hat{u}^h_{k,y} \bar{\phi}_y
                + b_1 \hat{u}^h_{k,y} \bar{\phi}
                + \dfrac{\lambda_{A,k}}{\lambda_{M,k}}
                \hat{u}^h_k \bar{\phi}
                dy
                =
                \int_0^L
                \hat{f}^h_k
                \bar{\phi}
                dy,\,
                \forall \phi \in V^{\Delta y}([0, L]).
            \end{equation}
            while the semidiscretization \(\hat{v}_k\) satisfies Equation
            \eqref{eq:cdr-2d-semidiscretization}. Hence Theorem
            \ref{thm:cdr-1d-theorem} bounds the difference between each
            \(\hat{u}^h_k - \hat{v}_k\) term, where
            \begin{equation}
                D =
                \dfrac
                {\sqrt{b_1^2 + \dfrac{4 \lambda_{A,k}}{\lambda_{M,k}}}}
                {2}
            \end{equation}
            has a positive real part since \(c > 0\). Since the eigenvalue ratio
            is bounded by Equations
            \eqref{eq:cdr-eigenvalue-ratio-lower-bound}%
            -\eqref{eq:cdr-eigenvalue-ratio-upper-bound},
            there exists a constant \(\kappa\) independent of \(k\), but
            dependent on \(b_1\) and \(c\), such that
            \begin{subequations}
                \begin{align}
                    1
                    &\leq
                    |D| \leq \kappa (1 + |k|)                                 \\
                    \dfrac{\BoundednessSymbol}{\sqrt{\CoercivitySymbol}}
                    &\leq
                    \kappa (1 + k^2)
                \end{align}
            \end{subequations}
            due to Equation \eqref{eq:cdr-eigenvalue-ratio-upper-bound} and the
            bound \(|k| \leq 1/(2 \Delta x)\). Let
            \begin{equation}
                C_7 = C_7(b_1, p, n, L) = C_5(b_1, 1, p, n, L)
            \end{equation}
            be the constant used in Theorem \ref{thm:cdr-1d-theorem} with \(m =
            1\). Then, applying Theorem \ref{thm:cdr-1d-theorem}, the error in
            the derivative in the boundary cell \(I_N\) is
            \begin{align}
                \left\|\dfrac{d^n}{dy^n}(\hat{v}_k - \hat{u}_k^h)\right\|
                _{\lastCellLinfy}
                \nonumber
                &\leq
                C_7 \bigg[|D|^{p + 5}
                \max\left(2, \dfrac{2}{|D| \Delta y}\right)
                \dfrac{\gamma^4}{\alpha^2}
                \|\hat{v}_k^{(2)}\|_{L^\infty} \Delta y^{3}                   \\
                &\phantom{\leq C_7 \bigg\lbrack} + |D|^2 \dfrac{\gamma^2}{\alpha}
                \|\hat{v}_k\|_{W^{2 + p, \infty}} \Delta y^{2 + p - n}
                \bigg]                                                        \\
                &\leq
                \nonumber
                C_7 \bigg[\kappa^{p + 9} (1 + |k|)^{p + 5} (1 + k^2)^4
                \max\left(2, \dfrac{2}{\Delta y}\right)
                \|\hat{v}_k^{(2)}\|_{L^\infty} \Delta y^3                     \\
                \label{eq:v-hat-k-dy-error-with-kappa}
                &\phantom{\leq C_7 \bigg\lbrack}
                + \kappa^4 (1 + |k|)^2
                (1 + k^2)^2
                \|\hat{v}_k\|_{W^{2 + p, \infty}} \Delta y^{2 + p - n}
                \bigg]
            \end{align}
            %% double check these norms
            By the fundamental theorem of calculus
            \begin{equation}
                \label{eq:v-k-hat-linf-regularity}
                \left\|
                \hat{v}_k^{(r)}
                \right\|_{L^\infty}
                \leq C([0, L])
                \left\|\hat{v}_k^{(r)}\right\|_{H^1}
                \leq C([0, L])
                \left\|\hat{v}_k\right\|_{H^{r + 1}}
            \end{equation}
            where \(C([0, L])\) is a constant dependent on the domain. Hence, by
            Equation \eqref{eq:v-k-hat-regularity-estimate}
            \begin{equation}
                \left\|
                \hat{v}_k
                \right\|_{W^{2 + p, \infty}}
                \leq
                C([0, L], 2 + p)
                \left(
                1 +
                |b_1| +
                \left|\dfrac{\lambda_{A,k}}{\lambda_{M,k}}\right|
                \right)
                C(f) k^{-q + p + 1}
            \end{equation}
            where \(C([0, L], 2 + p)\) is a constant dependent on both the
            domain and the polynomial degree. Combining the elliptic regularity
            estimate given by Equation \eqref{eq:v-k-hat-linf-regularity} and
            the error bound given by \eqref{eq:v-hat-k-dy-error-with-kappa}
            yields
            \begin{align}
                \nonumber
                \left\|\dfrac{d^n}{dy^n} (\hat{v}_k - \hat{u}_k)\right\|
                _{\lastCellLinfy}
                &\leq
                C_7
                %% common constants for both L-oo bounds
                C([0, L], 2 + p)
                \left(
                1 +
                |b_1| +
                %% factor of k^2
                \left|\dfrac{\lambda_{A,k}}{\lambda_{M,k}}\right|
                \right)
                C(f)
                %% sum of errors
                \bigg[                                                        \\
                \nonumber
                &\phantom{\leq C_7 \bigg\lbrack + }
                \kappa^{p + 9} (1 + |k|)^{p + 5} (1 + k^2)^4
                \max\left(2, \dfrac{2}{\Delta y}\right)
                k^{-q + 1}
                %% we have now k^{13 + p - q}
                \Delta y^3                                                    \\
                %% TODO add new label here \label{eq:v-hat-k-dy-error-with-kappa}
                &\phantom{\leq C_7 \bigg\lbrack}
                + \kappa^4 (1 + |k|)^2
                (1 + k^2)^2
                k^{-q + p + 1}
                \Delta y^{2 + p - n}
                %% we have now k^{7 + p - q}
                \bigg].
            \end{align}
            Note that the eigenvalue ratio \(\lambda_{A,k}/\lambda_{M,k}\)
            scales like \(O(k^2)\). Hence, by the assumption that \(q \geq 18 +
            p\) the summation over \(k\) converges:
            \begin{equation}
                \sum_{k=-\infty}^\infty
                \left|
                \dfrac{d^n}{dy^n}
                \left(
                \hat{v}_k(y) - \hat{u}_k^h(y)
                \right)
                \exp(2 \pi I k \delta_i)
                \right|
                = \CEightSymbolic (\Delta y^2 + \Delta y^{2 + p - n}).
            \end{equation}
            for a constant \(C_8\) independent of \(\Delta x\) and \(\Delta y\).
            Hence, by the triangle inequality
            \begin{equation}
                \left|\dfrac{d^n}{dy^n}(u - u^h)\right|
                _{(\delta_i, \delta_{N^*})}
                \leq
                \CSixSymbolic \Delta x^2
                +
                \CEightSymbolic
                \left(\Delta y^2 + \Delta y^{2 + p - n}\right)
            \end{equation}
            which is the desired result.
        \end{proof}

\section{Numerical Results}
    \label{sec:numerical-results}
    \subsection{Overview}
        This section summarizes numerical experiments verifying the rates
        of convergence proven in Theorems \ref{thm:cdr-1d-theorem} and
        \ref{thm:cdr-2d}. All experiments were performed with a finite element
        discretization of \eqref{eq:cdr-equation} in either one or two spatial
        dimensions, utilizing the \texttt{deal.II} library's \cite{dealii85}
        support for tensor product \textit{hp}-finite elements. For further
        information on algorithms and data structures for general \textit{hp}
        codes for continuous finite elements see \cite{BangerthHP}. The
        resulting linear systems were solved with the standard PETSc
        \cite{petsc-user-ref, petsc-efficient} GMRES linear solver and the
        \texttt{BoomerAMG} algebraic multigrid preconditioner from the HYPRE
        library \cite{hypre-web-page}. The preconditioner was configured to use
        SOR/Jacobi relaxation and Gaussian elimination for the coarse solve. The
        linear solver used a tolerance of \(10^{-14}\) times the Euclidean norm
        of the right-hand side vector. We verify the rates of convergence with
        respect to the seminorms defined by Equations
        \eqref{eq:h1-b-error-definition}-\eqref{eq:h2-b-error-definition} from
        Theorem \ref{thm:cdr-1d-theorem} and Theorem \ref{thm:cdr-2d} by
        performing uniform grid refinement studies.

    \subsection{1D Numerical Results}
        \label{subsec:numerical-results-1d}
        This subsection presents numerical verification of the convergence
        rates proven in Theorem \ref{thm:cdr-1d-theorem} for the seminorms
        defined by Equations
        \eqref{eq:h1-b-error-definition}-\eqref{eq:h2-b-error-definition}. We
        use the method of manufactured solutions to derive a forcing function
        for the exact solution
        \begin{equation}
            \label{eq:1d-exact-solution}
            u(x) = \sin(10 x)
        \end{equation}
        to Equation \eqref{eq:cdr-equation-oned} with \(b = 1\) and \(c = 2\).
        Figures \ref{fig:oned-h1-convergence}-\ref{fig:oned-h2-convergence}
        depict the errors in the pointwise boundary first and second derivative
        seminorms. The resulting finite element space is notated as
        \(Q^m{\hyphen}Q^{m + p}\), where \(m \in\{1, 2\}\) is the polynomial
        degree on interior cells and \(m + p\) (with \(p \in\{0,1,2,3\}\)) is
        the polynomial degree on boundary cells. These figures illustrate the
        two different rates of convergence for the boundary derivatives: The
        error in the boundary derivative depends both on the local polynomial
        degree (i.e., \(m + p\)) and on the approximation order at the last
        interior mesh vertex, which is (by Theorem
        \ref{thm:last-interior-vertex-superconvergence}) \(2 m + 1\). Hence, by
        Theorem \ref{thm:cdr-1d-theorem} the asymptotic convergence rate for the
        \(n\)th derivative should be \(\min(2 m, m + p - n + 1)\), which is what
        we observe in Figures
        \ref{fig:oned-h1-convergence}-\ref{fig:oned-h2-convergence}.
        \begin{figure}[!]
            \centering
            \includegraphics[width=3.2in]{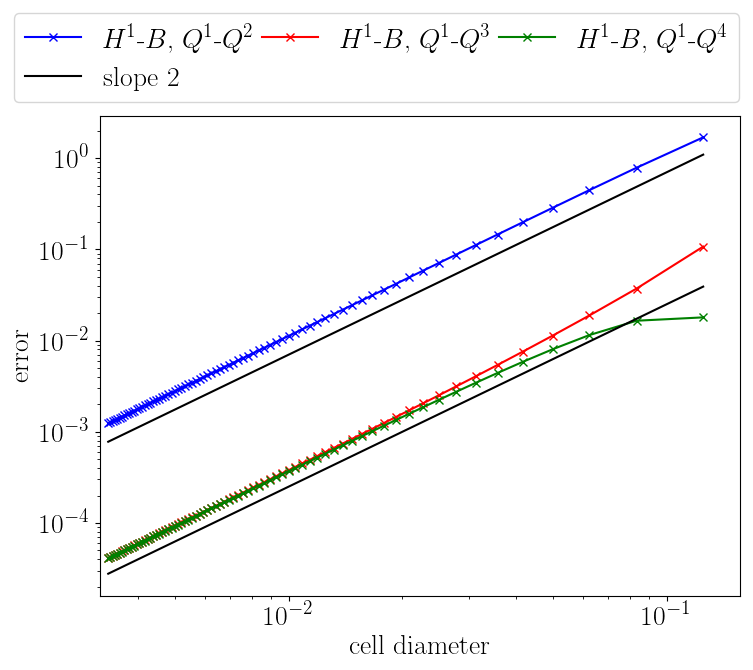}
            \includegraphics[width=3.2in]{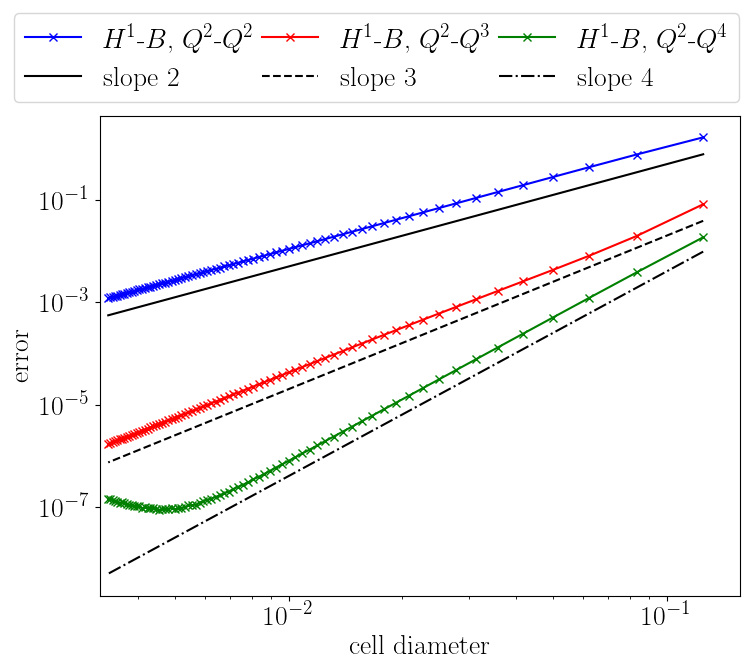}
            \caption
            {
            Rates of convergence in the \(H^1{\hyphen}B\) seminorm with exact
            solution \eqref{eq:1d-exact-solution}. The solution to the
            discretization using quadratic elements on the interior encounters
            roundoff error after sufficient grid refinement. These results
            verify the convergence rate proven in Theorem
            \ref{thm:cdr-1d-theorem}, i.e., the rate of convergence in the
            boundary derivatives is limited by the rate of convergence at the
            last interior mesh vertex and the polynomial degree on the boundary
            cell.
            }
            \label{fig:oned-h1-convergence}
        \end{figure}

        \begin{figure}[!]
            \centering
            \includegraphics[width=3.2in]{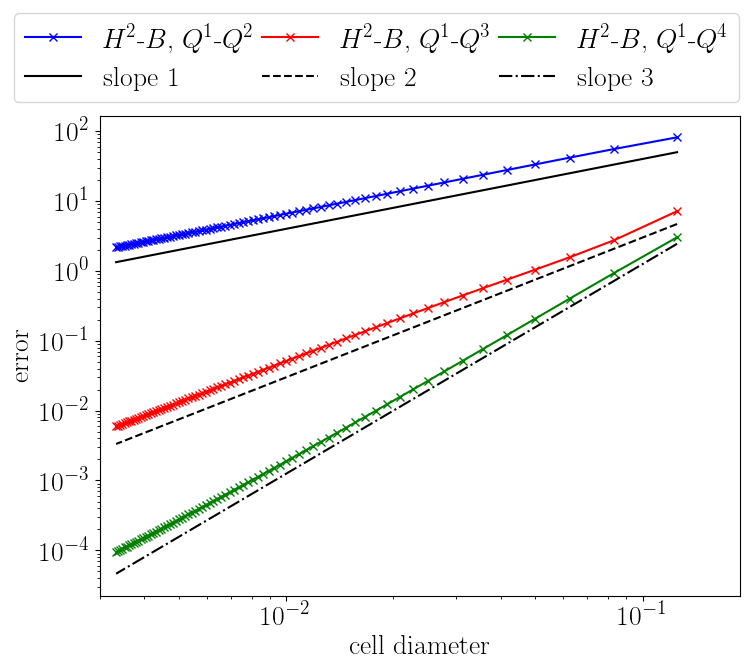}
            \includegraphics[width=3.2in]{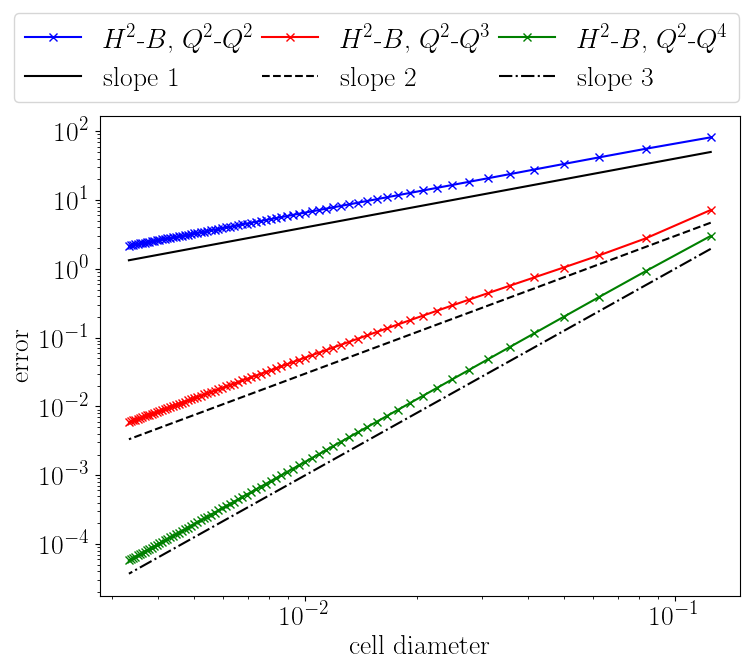}

            \caption{Rates of convergence in the \(H^2{\hyphen}B\) seminorm with
            exact solution \eqref{eq:1d-exact-solution}. The second boundary
            derivative error with the \(Q^1{\hyphen}Q^2\), \(Q^1{\hyphen}Q^3\),
            and \(Q^1{\hyphen}Q^4\) discretizations is dominated by the local
            error instead of the coupling error, resulting in third order
            convergence for sufficient boundary \(p\)-refinement.}
            \label{fig:oned-h2-convergence}
        \end{figure}

    \subsection{2D Numerical Results}
        This subsection presents numerical verification of the convergence
        rates proven in Theorem \ref{thm:cdr-2d} (i.e., derivative convergence
        rates for interior bilinear elements and a periodic boundary condition
        along the non-enriched boundaries). Additional numerical
        experiments demonstrate that these improved convergence rates can be
        obtained in other geometries, such as a square with only Dirichlet
        boundary conditions and a disk. In all cases, the discretization
        consists of bilinear elements on all interior cells and \(p\)-refinement
        limited to boundary cells. All numerical experiments, unless
        otherwise noted, use anisotropic (i.e., only in the normal direction)
        \(p\)-refinement on boundary cells with Dirichlet boundary conditions.
        We use the method of manufactured solutions to derive forcing functions.
        All test problems use \(\vec{b} = (1, 1)\) and \(c = 2\).

        \subsubsection{The Necessity of Normal \texorpdfstring{\(p\)}{\emph{p}}-Refinement}
            An important part of the proof for Theorem \ref{thm:cdr-2d} is
            translation-invariant property of the discretization in the periodic
            direction. Numerical experiments, summarized in Figure
            \ref{fig:nonnormal-p-refinement-convergence}, indicate that this is
            not merely a convenient assumption: in the case of interior bilinear
            elements, the discretization must be equivalent to a tensor product
            of two one-dimensional discretizations in order to obtain the
            improved convergence rates in the seminorms defined by Equations
            \eqref{eq:h1-b-error-definition}-\eqref{eq:h2-b-error-definition}.

            To demonstrate this, we consider discretizations with interior
            bilinear elements and either \emph{normal \(p\)-refinement} or
            \emph{isostropic \(p\)-refinement}, where the second case still uses
            continuity constraints to obtain a conforming solution. Notate
            the polynomial space on each boundary cell, \(P^m \otimes P^{m +
            p}\), by \(Q^{(m, m + p)}\), where \(m\) is the degree in the
            tangential direction and \(m + p\) is the degree in the normal
            direction. These numerical experiments use the manufactured
            solution
            \begin{equation}
                \label{eq:periodic-2d-solution}
                u_1(x, y) = (y^3 + \exp(-y^2) + \sin(4.5 y^2) + \sin(20 y))
                (20 \cos(4 \pi x) + 0.1 \sin(20 \pi x) - 80 \sin(6 \pi x)).
            \end{equation}

            \begin{figure}[!]
                \centering
                \includegraphics[width=3.2in]{periodic-nonnormal-convergence.png}
                \includegraphics[width=3.2in]{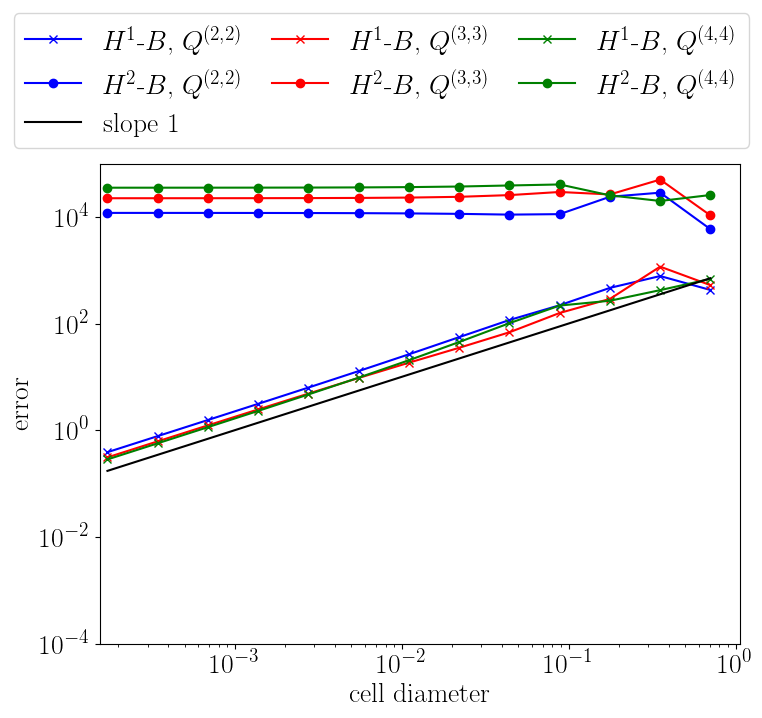}

                \caption{Rates of convergence for a solution to Equation
                \eqref{eq:cdr-equation} with periodic boundary conditions
                in the \(x\) direction. One set of discretizations uses
                \(p\)-refinement in just the normal direction (left, with the
                scheme proposed by Figure \ref{fig:show-nonnormal-elimination})
                and the other uses isotropic \(p\)-refinement (right). The
                boundary cell polynomial spaces are notated as \(Q^{(m, m +
                p)}\). These results indicate that the tensor product structure
                assumed in Theorem \ref{thm:cdr-2d} is necessary.}
                \label{fig:nonnormal-p-refinement-convergence}
            \end{figure}

            The results in Figure \ref{fig:nonnormal-p-refinement-convergence}
            show that the tensor-product structure of the discretization used in
            Theorem \ref{thm:cdr-2d} is necessary for achieving improved
            boundary derivative convergence rates (i.e., order \(2\) for first
            derivatives and order \(1\) or \(2\) for second derivatives).
            Performing isotropic \(p\)-refinement (which results in a larger
            finite element space then \(p\)-refinement purely in the normal
            direction) does not improve the rates of convergence even though the
            approximation space is larger.

        \subsubsection{Extension to a Nonperiodic Boundary}
            \label{subsec:nonperiodic-numerical-results}
            Theorem \ref{thm:cdr-2d} only applies to domains with periodic
            boundary conditions in one of the coordinate directions; however,
            numerical experiments indicate that the normal \(p\)-refinement
            scheme improves derivative convergence rates in more general
            geometries. To demonstrate this, we performed numerical experiments
            on a rectangular domain with Dirichlet boundary conditions, a
            \(Q^{(1 + p, 1 + p)}\) element in each corner (which corresponds to
            the tensor product of normal refinement in both directions), and
            \(Q^{(1, 1 + p)}\) elements in the other boundary cells. Like the
            other 2D numerical experiments, we only consider bilinear elements
            (\(Q^{(1, 1)}\)) on the interior of the domain. The manufactured
            solution in this test case, which is not periodic in either
            direction, is
            \begin{equation}
                \label{eq:nonperiodic-2d-solution}
                u_2(x, y) = x y \sin(20 y) + 10 \exp(-x y) \cos(15 x) + 2
                \sin(10^{y^2 + \cos(x)}) + \sin(30 x y).
            \end{equation}
            Figure \ref{fig:square-nonperiodic-domain} summarizes the
            results of these experiments. These experiments show that the
            results proven in Theorem \ref{thm:cdr-2d} generalize to a domain
            with purely Dirichlet boundary conditions.

            \begin{figure}[!]
                \centering
                \includegraphics[width=3.2in]{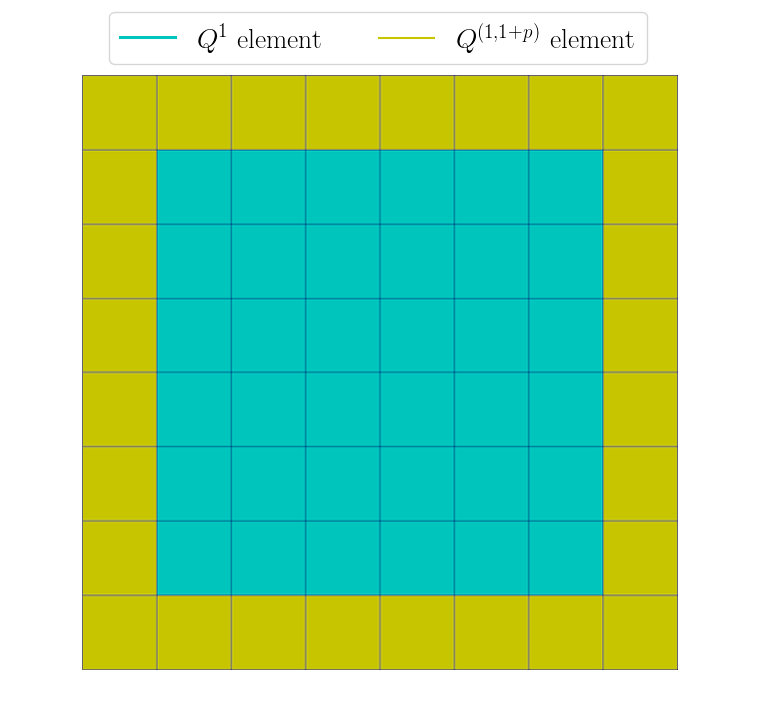}
                \includegraphics[width=3.2in]{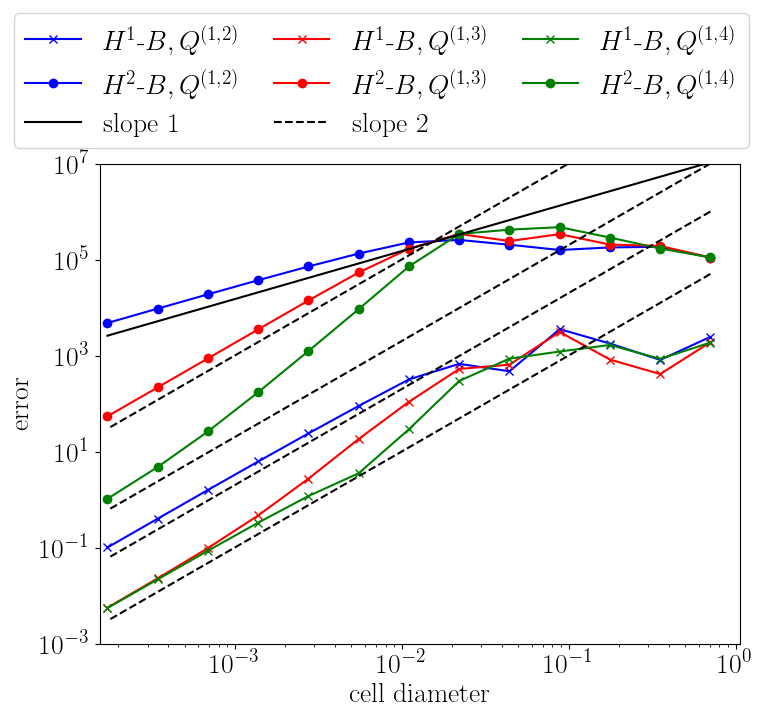}
                \caption{Depiction of the nonperiodic square domain and
                numerical convergence rates. The picture on the left shows
                which cells have been \(p\)-refined in the normal
                direction after three global grid refinements. The picture on
                the right shows the rates of convergence. This experiment
                shows that the results proven in Theorem \ref{thm:cdr-2d} hold
                when all boundaries are Dirichlet and the discretization has a
                tensor-product structure.}
                \label{fig:square-nonperiodic-domain}
            \end{figure}

        \subsubsection{Extension to a Disk With Radial \texorpdfstring{\(p\)}{\emph{p}}-Refinement}
            \label{subsec:circle-numerical-results}
            The final numerical example shows that the convergence rate proven
            in Theorem \ref{thm:cdr-2d} holds in a non-Cartesian geometry.
            Consider a disk whose central cells are aligned with the \(x,
            y\) axes, cells near the boundary are aligned with the \(r, \theta\)
            axes (i.e., they are rectangles in polar coordinates), and cells
            between these two regions are geometrically described with a
            transfinite interpolation between the Cartesian and polar regimes.
            This geometry description (i.e., Cartesian coordinates in the
            center, polar coordinates at the boundary, and a transfinite
            interpolation in between) results in a well-conditioned grid with
            boundary cells aligned with the polar coordinate axes after mesh
            refinement. A picture of the grid after one refinement is shown in
            Figure \ref{fig:circular-domain-picture}. Since the order of
            convergence under the seminorms defined by Equations
            \eqref{eq:h1-b-error-definition}-\eqref{eq:h2-b-error-definition} is
            second-order, we use the standard second-order bilinear mapping from
            the reference cell to the physical cell to perform all cell
            calculations. The manufactured solution in this test case is
            Equation \eqref{eq:periodic-2d-solution}.

            \begin{figure}[!]
                \centering
                \includegraphics[width=3.2in]{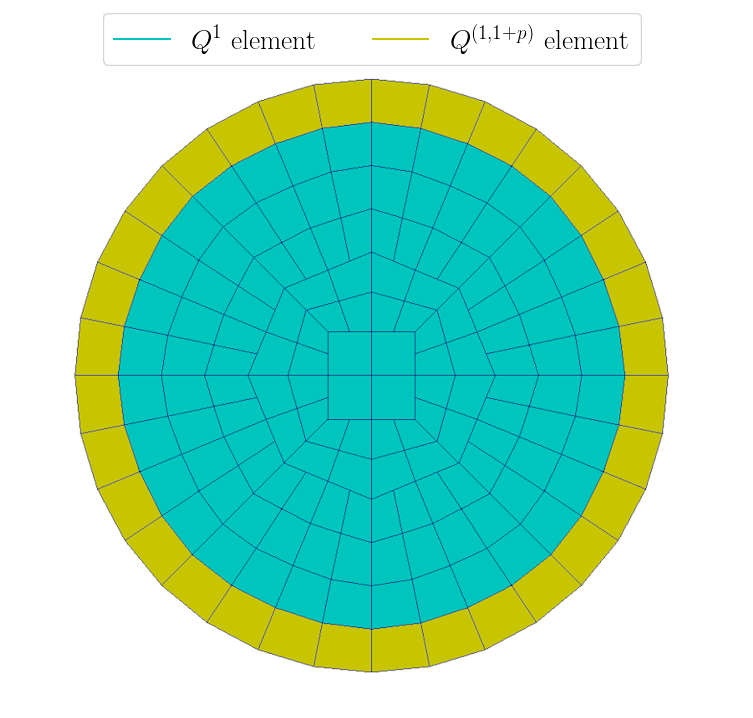}
                \includegraphics[width=3.2in]{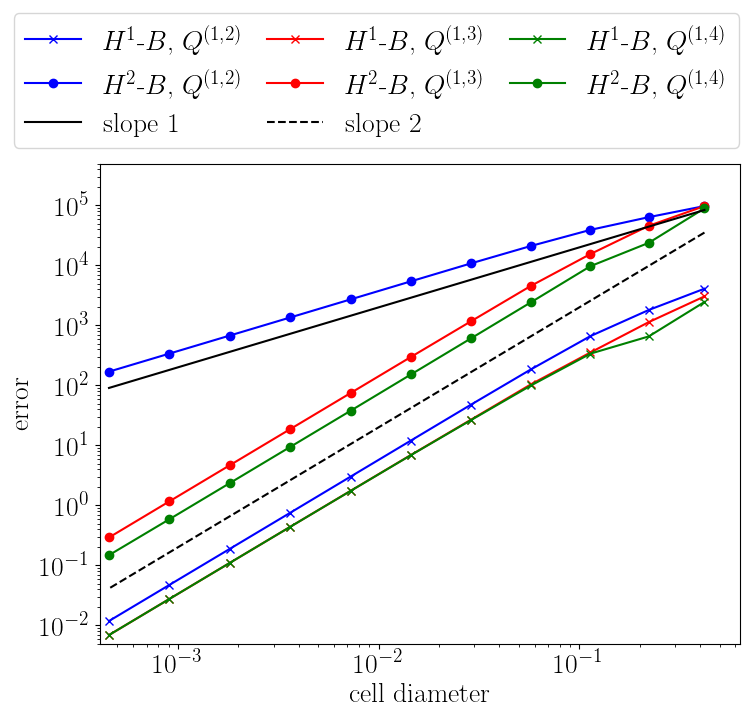}
                \caption{Depiction of the disk grid and numerical results. The
                picture on the left shows which cells have radial
                \(p\)-refinement after one global grid refinement and the
                picture on the right shows the rates of convergence: The
                \(H^1{\hyphen}B\) error plot for \(Q^{(1,3)}\) is obscured by
                the plot for \(Q^{(1,4)}\). These results show that the
                convergence rates proven in Theorem \ref{thm:cdr-2d} apply to
                more general settings where the cells are aligned with a smooth
                boundary.}
                \label{fig:circular-domain-picture}
            \end{figure}

            Based on the results in Figure \ref{fig:circular-domain-picture}, we
            conjecture that the convergence rate proven in Theorem
            \ref{thm:cdr-2d} holds for sufficiently regular grids where the
            cells near the boundary are aligned with the boundary itself (i.e.,
            the cell faces are either orthogonal or parallel to the boundary).
            This implies that the proposed \(p\)-refinement strategy will
            improve boundary derivative convergence rates when the boundary of a
            domain is sufficiently smooth.

\section{Concluding Remarks}
    \label{sec:conclusions}
    This work proposed a new method for achieving higher-order accuracy in
    boundary derivative calculations while still using lower-order finite
    elements on the interior of the domain. The method, in essence, adds new
    degrees of freedom that enrich the boundary finite elements in a direction
    normal to the boundary itself. The numerical experiments imply that
    isotropic \(p\)-refinement does not improve the order of accuracy of
    boundary derivatives: put another way, the finite element space must have a
    tensor-product structure to obtain higher-order derivative convergence on
    the boundary. The proposed method resembles a finite difference
    discretization in the sense that the higher rate of convergence is available
    at all boundary vertices and that the rate of convergence in the first and
    second derivatives, after performing at least two levels of
    \(p\)-refinement, is equal to the global \(L^\infty\) rate of convergence.

    There are several possible extensions of this work: one could derive a
    similar result to Theorem \ref{thm:cdr-2d} by interpreting higher-order
    (e.g., biquadratic) basis functions as finite difference methods with
    nonuniform stencils. Another possible direction is finding a solid
    theoretical backing for the numerical examples given in Subsection
    \eqref{subsec:nonperiodic-numerical-results} and Subsection
    \eqref{subsec:circle-numerical-results}, which show results for more general
    geometries than a square with periodic boundary conditions in the \(x\)
    direction. There are possibilities for extending the implementation as
    well: instead of Lagrange \(p\)-refinement, one could add degrees of freedom
    corresponding to normal derivatives on the boundary. Finally, since this
    method does not rely on any solution postprocessing procedures, it may be a
    useful technique to use in applications where one requires higher accuracy
    in the derivatives of a solution on a particular boundary and an energy
    estimate, such as complex boundary conditions for multiphysics applications.

\begin{appendices}
\section{Inequalities used to bound the Greens' function}
    \label{sec:greens-inequalities}
    \begin{enumerate}
        \item \begin{equation}
                    \begin{aligned}
                        \left|\dfrac{\exp(- z) - 1}{z}\right| &\leq 1
                    \end{aligned}
                \end{equation}
                as, if \(\mathrm{Im}(z) = 0\), then the function is
                maximized as \(z \rightarrow 0^+\), has negative slope for
                \(z \geq 0\), and is bounded below by \(0\). If
                \(\mathrm{Im}(z) \neq 0\) then by the maximum modulus
                principle this function is bounded above by the value of
                the analytic continuation along the boundary
                \(\mathrm{Re}(z) = 0\), which is also \(1\). This implies
                that
                \begin{equation}
                    \left|
                    \dfrac{\exp(-2 D \Delta x) - 1}{2 D}
                    \right|
                    \leq \Delta x.
                \end{equation}
          \item \begin{subequations}
                    \begin{align}
                        \left|
                        \dfrac{\exp(2 D L) \pm \exp(2 D \Delta x)}
                        {\exp(2 D L) - 1}
                        \right|
                        &\leq
                        2 \left|
                        \dfrac{\exp(2 D L)}{\exp(2 D L) - 1}
                        \right|                                               \\
                        &\leq
                        2
                        \dfrac{e}{e - 1}                                      \\
                        &\leq
                        4
                    \end{align}
                \end{subequations}
                since \(\mathrm{Re}(D)\) and \(L\) are greater than unity.
        \item If \(L - \Delta x \leq x \leq L\) then
              \begin{subequations}
                \begin{align}
                        \left|
                        \dfrac
                        {
                        \exp(-D(L + \Delta x - x)) \pm
                        \exp(D(L - \Delta x - x))
                        }
                        {2 D}
                        \right|
                        &\leq
                        \left|
                        \dfrac
                        {
                        \exp(-D(L + \Delta x - x))
                        }
                        {2 D}
                        \right|
                        +
                        \left|
                        \dfrac
                        {
                        \exp(D(L - \Delta x - x))
                        }
                        {2 D}
                        \right|                                               \\
                        &\leq \dfrac{1}{|D|}
                \end{align}
              \end{subequations}
              since \(0 \leq L + \Delta x - x\) and \(L - \Delta x - x
              \leq 0\), so the arguments of both exponentials have
              negative real parts.
    \end{enumerate}

\section{Bounding the ratio of \texorpdfstring{\(\exp\)}{exp} and \texorpdfstring{\(\sinh\)}{sinh}}
    \label{sec:appendix-exp-sinh-ratio}
    for \(z = a + b I\) and \(a > 0\):
    \begin{subequations}
        \begin{align}
            |\sinh(z)|^2
            &=
            \left(\cos(b) \cosh(a)\right)^2 + \left(\sin(b) \sinh(a)\right)^2 \\
            &\geq
            \left(\cos(b) \sinh(a)\right)^2 + \left(\sin(b) \sinh(a)\right)^2 \\
            &=
            \left(\sinh(a)\right)^2                                           \\
            &\geq a^2
        \end{align}
    \end{subequations}
    Hence, if \(|b| \leq |a|\)
    \begin{equation}
        2 |\sinh(z)| \geq |z|.
    \end{equation}
    Hence
    \begin{subequations}
        \begin{align}
            \dfrac{\exp(z)}{\sinh(z)}
            &= 2 + \dfrac{\exp(-z)}{\sinh(z)}                                 \\
            \left|\dfrac{\exp(z)}{\sinh(z)}\right|
            &= 2 + \left|\dfrac{\exp(-z)}{\sinh(z)}\right|                    \\
            &\leq
            2 + \dfrac{2}{|z|}
        \end{align}
    \end{subequations}
\end{appendices}

% \begin{thebibliography}
% \bibliographystyle{plain}
% \bibliography{prefine}
% \end{thebibliography}
\bibliographystyle{plain}
\bibliography{paper}
\end{document}